\documentclass[11pt,a4paper]{article}
\usepackage[utf8]{inputenc}
\usepackage{graphicx}
\usepackage{setspace}
\usepackage{amsthm}
\usepackage{float}
\usepackage{subfig}
\usepackage{xr}
\usepackage{bm}
\usepackage[pdftex=true,hyperindex=true,colorlinks=true]{hyperref}
\usepackage{mathtools}
\usepackage{indentfirst}
\usepackage[T1]{fontenc}
\usepackage{color}
\usepackage{amsfonts} %on charge les fontes double-barres et gothiques
\usepackage{amsmath}

\usepackage{mathrsfs}
\usepackage{tikz}
\usepackage{dsfont}
\usetikzlibrary{calc}
\usetikzlibrary{patterns}
\usepackage{pstricks}
\usepackage{esint}
\tikzset{math3d/.style={x={(-0.353cm,-0.353cm)},z={(0cm,1cm)},y={(1cm,0cm)}}}
\newtheoremstyle{remark1}% name of the style to be used
  {\topsep}   % ABOVESPACE
  {\topsep}   % BELOWSPACE
  {\normalfont}  % BODYFONT
  {0pt}       % INDENT (empty value is the same as 0pt)
  {\bfseries} % HEADFONT
  {.}         % HEADPUNCT
  {5pt plus 1pt minus 1pt} % HEADSPACE
  {}          % CUSTOM-HEAD-SPEC

\newtheorem{Theo}{Theorem}[section] %ou :
\newtheorem{Prop}{Proposition}[section]
\newtheorem{Lem}{Lemma}[section]
\newtheorem{Cor}{Corollary}[section]

\newtheorem{Def}{Definition}[section]

\theoremstyle{remark1}
\newtheorem{Rem}{Remark}[section]
\newenvironment{Dem}[1]
{\begin{trivlist}\item[\hskip5mm{\sc #1}]}{\end{trivlist}}

\newcommand{\R}{\mathbb{R}}

\newcommand{\dx}{\ \textrm{d}x}

\newcommand{\harpoon}{\ -\hspace{-0.2cm}\rightharpoonup \  }
\newcommand{\T}{\textrm{T}}
\newcommand{\Div}{\textrm{Div}}
\newcommand{\e}{\varepsilon}
\newcommand{\N}{\mathbb N}
\renewcommand{\qed}{\hfill \ensuremath{\Box}}
\newcommand{\supp}{\operatorname{supp}} 
\renewcommand{\div}{\textrm{div}}
\newcommand{\curl}{\textrm{curl}}

\def\Sum{\displaystyle\sum}
\newpage
\usepackage{tikz}
\usetikzlibrary{patterns}
\usetikzlibrary{arrows} 

\tikzset { xmin/.store in=\xmin, xmin/.default=-3, xmin=-3,
           xmax/.store in=\xmax, xmax/.default=3, xmax=3,
           ymin/.store in=\ymin, ymin/.default=-3, ymin=-3,
           ymax/.store in=\ymax, ymax/.default=3, ymax=3}

\tikzset {domaine/.style 2 args={domain=#1:#2}}

\newcommand{\eq}{\begin{equation}}
\newcommand{\qe}{\end{equation}}
\newcommand{\eqs}{\begin{equation*}}
\newcommand{\qes}{\end{equation*}}

\makeatletter
\renewcommand{\thefigure}{\ifnum \c@section>\z@ \thesection.\fi
 \@arabic\c@figure}
\@addtoreset{figure}{section}
\makeatother
\makeatletter
\renewcommand{\theequation}{\ifnum \c@section>\z@ \thesection.\fi
 \@arabic\c@equation}
\@addtoreset{equation}{section}
\makeatother
\title{ Magneto-resistance in three-dimensional composites.}
\author{\begin{tabular}{cc}
    {Marc BRIANE } & {Laurent PATER\footnote{Corresponding author.}}
    \\*[-.0em]
    {\small  Institut de Recherche Math\'ematique de Rennes} & {\small Institut de Recherche Math\'ematique de Rennes}
    \\*[-.3em]
    {\small  INSA de Rennes} &	 {\small Universit\'e de Rennes 1}
    \\*[-.3em]
    {\small mbriane@insa-rennes.fr} & {\small laurent.pater@ens-cachan.org}
\end{tabular}}
\begin{document}
\maketitle
\textcolor{white}{mon texte}
\begin{abstract}
In this paper we study the magneto-resistance, \textit{i.e.} the second-order term of the resistivity perturbed by a low magnetic field, of a three-dimensional composite material. Extending the two-dimensional periodic framework of \cite{Brimagneto}, it is proved through a $H$-convergence approach that the dissipation energy induced by the effective magneto-resistance is greater or equal to the average of the dissipation energy induced by the magneto-resistance in each phase of the composite. This inequality validates for a composite material the Kohler law which is known for a homogeneous conductor. The case of equality is shown to be very sensitive to the magnetic field orientation. We illustrate the result with layered and columnar periodic structures. 
\end{abstract}
\maketitle

\section{Introduction}

In a conductor with a matrix-valued resistivity $\rho$, a low magnetic field $h \in \R^3$ induces a perturbed resistivity $\rho(h)$. Due to Onsager relations (see \cite{LaLi,Ali}), the perturbed resistivity satisfies \eq\label{intro0} \rho(h) = \rho(-h)^\T. \qe As a consequence, the perturbed resistivity admits the following second-order expansion (see Section~\ref{sect1}): \eq \rho(h) = \rho(0) + \mathscr R( h) + \mathscr M(h,h) + o(|h|^2), \qe where $\rho(0)$, $\mathscr M(h,h)$ are symmetric matrices and $\mathscr R(h)$ is an antisymmetric matrix. On the one hand, according to the Hall effect (see, \textit{e.g.}, \cite{LaLi}), the magnetic field induces a transversal electric field $E_t(h)$  which balances the magnetic force acting on the charge carrier and is perpendicular to the current $j$. It is given by  \eq E_t(h) = \mathscr R(h) \, j \, \bot \, j, \qe where $\mathscr R(h)$ is the Hall tensor which reduces to $r (j \times h)$ in the isotropic case. On the other hand, the so-called magneto-resistance $\mathscr M(h,h)$ measures the difference between the perturbed dissipation energy and the unperturbed one, namely\eq \rho(h) j \cdot j - \rho(0) j\cdot j = \mathscr M(h,h) j \cdot j + o(|h|^2), \label{introexpmag}\qe
in which the Hall term plays no role (due to the antisymmetry of $\mathscr R(h)$). Expansion \eqref{introexpmag} has to be regarded in connection to the Kohler law \cite{Koh} which states that the symmetrized resistivity (without the Hall term) $\rho_s(h)$ of a homogeneous conductor satisfies the asymptotic \eq \rho_s(h) - \rho(0) \underset{h \to 0}{\approx}  m \, |h|^2 \quad \text{with } m > 0, \label{Kohlaw}\qe which corresponds to an increase of the magneto-resistance.

When the conductor has a microstructure characterized by a scale $\e$, the resistivity $\rho_\e(h)$ depends on the two parameters $\e,h$. In the framework of the Murat Tartar $H$-convergence theory (see Section~\ref{sect1} and \cite{MurComp,tpeccot}), the conductivity $\sigma_\e(h) = \rho_\e(h)^{-1}$ $H$-converges to the effective (or homogenized) conductivity $\sigma_*(h)$. Under appropriate regularity conditions for $\sigma_\e(h)$ (see \eqref{cond-reg}), the effective resistivity $\rho_*(h) = \sigma_*(h)^{-1}$ also satisfies equality \eqref{intro0} and the second-order expansion \eq \rho_*(h) = \rho_*(0) + \mathscr R_*( h) + \mathscr M_*(h,h) + o(|h|^2), \label{expintro1}\qe where $\mathscr R_*$ is the effective Hall tensor and $\mathscr M_*$ is the effective magneto-resistance tensor.

\bigskip

In his seminal work \cite{Ber}, Bergman gave for a periodic composite material an expression of the effective Hall matrix in terms of the local Hall matrix and the local current fields obtained in the absence of a magnetic field. Bergman's approach was extended in dimension two \cite{BMM} and in dimension three \cite{BM1} in the non-periodic framework of $H$-convergence. 

In dimension two, the conductor lies in a plane $(e_1,e_2)$ embedded in a transversal magnetic field $h \, e_3$, so that the local/effective Hall coefficient $r_{\e/*}$ and the local/effective magneto-resistance matrix $M_{\e /*}$ are defined by \eq \mathscr R_{\e/*}(h) = r_{\e / *} \, h \, \left( \begin{smallmatrix} 0 & -1 \\ 1 & 0 \end{smallmatrix} \right)\quad \text{and} \quad \mathscr M_{\e / *}(h,h) = h^2 \, M_{\e / *}. \qe In the periodic case, \textit{i.e.} when $r_\e(x) = r(x / \e)$ and $M_\e(x) = M(x / \e)$ are oscillating functions of the fast variable $x/\e$, it was proved in \cite{Brimagneto} that \eq M_* \, \langle j \rangle \cdot \langle j \rangle - \left \langle M \, j \cdot j\right \rangle \geq 0,\label{introineq2}\qe for any unperturbed current $j$, and that \eqref{introineq2} is an equality if and only if the Hall coefficient is a constant. By the Kohler law \eqref{Kohlaw}, the magneto-resistance in each phase satisfies  $M = \mu \, I_2$ with $\mu >0$, which implies the positivity of $M_*$ by \eqref{introineq2}. Then, the positivity of $m$ in \eqref{Kohlaw} corresponds to the positivity of $M_*$ in \eqref{introineq2}. Therefore, the inequality \eqref{introineq2} extends the classical Kohler law \eqref{Kohlaw} to anisotropic two-dimensional composites (see \cite{Brimagneto}, Remark 2.6).

\bigskip

This paper extends the results of \cite{Brimagneto} to three-dimensional composites. In dimension three, the local/effective Hall tensor reads as \eq \mathscr R_{\e / *} \cdot h = \mathscr E(R_{\e / *} \,h), \quad \text{with } \mathscr E(\eta):= \left(\begin{smallmatrix} 0 & -\eta_1 & \eta_2 \\ \eta_1 & 0 & - \eta_3 \\ -\eta_2 & \eta_3 & 0\end{smallmatrix}\right),\label{defR}\qe where $R_{\e / *}$ is called the local/effective Hall matrix. First, we obtain a general expression (see Theorem \ref{THconv}) for the difference between the effective dissipation energy due to the magneto-resistance and the average of the local dissipation energy. Then, extending a classical duality principle (see Lemma \ref{divcurl}), we prove that this difference is non-negative (see Theorem \ref{Thineq}), and equal to zero if and only if the Hall matrix satisfies some compactness condition. In the periodic case, this reads as (see Corollary \ref{Thineqper}) \eq D(h,h):=\mathscr M_*(h,h) \, \langle j \rangle \cdot \langle j \rangle - \left \langle \mathscr M(h,h) \, j \cdot j\right \rangle \geq 0,\label{introD}\qe for any unperturbed current field $j$. Moreover, \eqref{introD} is an equality if and only if \eq \mathrm{Curl} \left( \mathscr{E}(R \, h) j \right) =0 \quad \text{in } \mathscr D'(\R^3). \label{introcond}\qe 

We also investigate the behaviour of higher even-order terms. We show that inequality \eqref{introD} reverses when the magneto-resistance is replaced by the fourth-order term and the Hall matrix is assumed to be zero (see Proposition \eqref{4}). However, the equivalent of $D(h,h)$ for even-order term higher or equal to $4$ may have both a positive and a negative eigenvalue, so that \eqref{introD} cannot be extended (see Proposition \ref{cx}).

Then, the condition of equality \eqref{introcond} is discussed in the case of columnar  composites. First, an explicit formula for the difference of dissipation energies $D(h,h)$ \eqref{introD} is given (see Proposition~\ref{onedirection}) for a periodic material which is layered in some direction $\xi$. Second, for a general columnar structure in the direction $e_3$, the equality $D(h,h) = 0$ is shown to be very sensitive to the orientation of the magnetic field (see Proposition \ref{cyl}). More precisely, the equality $D(h,h) = 0$ implies that $\sigma(y_1,y_2)$ is a tensor product of type $f( h_1 y_1 + h_2 y_2) \, g (h_2 y_1 - h_1 y_2)$. For example, in the case of a four-phase checkerboard $\alpha_1, \alpha_2, \alpha_3, \alpha_4$ (see figure \ref{check1} below), we obtain that for any magnetic field $h \neq 0$ perpendicular to $e_3$, $D(h,h) \neq 0$ if $\alpha_1 \, \alpha_3 \neq \alpha_2 \, \alpha_4$ (see Proposition \ref{check}).

\bigskip

The paper is organized as follows. In Section \ref{sect1}, we recall results on the homogenization of the Hall effect and the magneto-resistance in order to establish an asymptotic formula for the effective magneto-resistance. In Section \ref{sect2}, we prove inequality \eqref{introD} and deal with the case of higher-order terms. Section \ref{exemples} is devoted to the case of equality for layered and columnar composites.

\subsection*{Notations}

\begin{itemize}
\item[$\bullet$] $|\cdot|$ denotes the euclidean norm in $\R^d$ for any positive integer $d$ and $(e_1, \ldots , e_d)$ the canonic basis of $\R^d$.
\item[$\bullet$] $\times$ denotes the cross product and $\otimes$ the tensor product in $\R^3$.
\item[$\bullet$] $\R^{d \times d}$ denotes the set of the real-valued ($ d \times d$) matrices and $I_d$ denotes the unit matrix of $\R^{ d \times d}$. 
\item[$\bullet$] $\R_a^{d \times d}$ (resp. $\R_s^{d \times d}$) is the set of the real-valued ($ d \times d$) antisymmetric matrices (resp. symmetric matrices). 
\item[$\bullet$] $A^s$ denotes the symmetric part of $A$, $A^\T$ its transposed matrix, and $\text{Cof}(A)$ its cofactors matrix.
\item[$\bullet$] $\Omega$ denotes a bounded open set of $\R^d$.
\item[$\bullet$] For $\alpha, \beta > 0$, $\mathcal M(\alpha,\beta ; \Omega)$ denotes the set of the invertible matrix-valued functions $A$ measurable in $\Omega$ and such that \eq \forall \ \xi \in \R^d, \quad A(x) \xi \cdot \xi \geq  \alpha |\xi|^2 \quad \text{and} \quad A(x)^{-1} \xi \cdot \xi \geq \beta^{-1} |\xi|^2 , \quad \text{a.e. } x \in \Omega. \label{M}\qe
\item[$\bullet$] For a vector-valued function $U : \Omega \longrightarrow \R^d$, \eqs DU :=\left[ \frac{\partial U_j}{\partial x_i}\right]_{1 \leq i,j \leq d} \quad \text{,} \quad \div\left(U \right) = \Sum_{i=1}^d \frac{\partial U_i}{\partial x_i} \quad \text{and} \quad \mathrm{curl}\left(U \right) = \left( \frac{\partial U_{i}}{\partial x_j} -  \frac{\partial U_j}{\partial x_i}\right)_{1 \leq i,j \leq d}. \qes
\item[$\bullet$] For a matrix-valued function $\Sigma : \Omega \longrightarrow \R^{d \times d}$, \eqs \Div\left(\Sigma \right) = \left( \Sum_{i=1}^d \frac{\partial \Sigma_{i,j}}{\partial x_i} \right)_{1 \leq j \leq d} \quad  \text{and} \quad \mathrm{Curl}\left(\Sigma \right) = \left( \frac{\partial \Sigma_{i,k}}{\partial x_j} -  \frac{\partial \Sigma_{j,k}}{\partial x_i}\right)_{1 \leq i,j,k \leq d}. \qes
%\item[$\bullet$] For a vector $v$, the operator $\partial_v$ denotes the derivation into the direction $v$. 
\item[$\bullet$] If $H$ is a vector space endowed with a norm $\Vert \cdot \Vert$, the equality $g_\e(h) = o_H(|h|^n)$, for $n \in \mathbb N$, means that \eq \lim \limits_{h \to 0} \left( \frac{1}{|h|^n} \sup \limits_{\e > 0} \Vert g_\e(h)\Vert \right)=0. \label{def-o}\qe
\item[$\bullet$] For $k \in \mathbb N$, $\mathscr C_c^k(\Omega)$ denotes the space of $k$-continuously derivable functions with compact support in $\Omega$.
\item For $\nu = (\nu_1, \ldots , \nu_d) \in \mathbb N^d$, we denote \eqs |\nu| = \nu_1 + \cdots + \nu_d \quad \text{and}\quad \displaystyle \frac{\partial^{|\nu|} }{\partial x^\nu} = \frac{\partial^{\nu_1} }{\partial x_1^{\nu_1}} \cdots \frac{\partial^{\nu_d} }{\partial x_d^{\nu_d}} .\qes 
\item[$\bullet$] $Y:=(0,1)^3$, and the $Y$-average is denoted $\langle \cdot \rangle$.
\item[$\bullet$] $H_\sharp(Y;Z)$ denotes the space of the $Y$-periodic functions from $\R^3$ to $Z$ which belong to $H_{\text{loc}}(\R^d)$ for a generic function space $H$.

\end{itemize}

\begin{Rem}
Consider a sequence $g_\e(h)$ in $H$ which satisfies the expansion of order $n \in \N$,
\eq g_\e(h) = g_\e^0 + g_\e^1(h) + \cdots + g_\e^n(h,\ldots,h) + o_{H}(|h|^n), \qe where for any $k \leq n$, $h \mapsto g_\e^k(h, \ldots, h)$ are bounded sequences of $k$-linear symmetric forms in $H$. In view of \eqref{def-o} each term $g_\e^k(h,\ldots,h)$ of the expansion inherits of the same (weak or strong) convergence of $g_\e(h)$ in $H$. \label{Remconv}
\end{Rem}

Let us recall the definition of the $H$-convergence due to Murat, Tartar \cite{tpeccot}:\begin{Def}[Murat, Tartar \cite{tpeccot}]
A sequence $A_\e$ in $\mathcal M(\alpha,\beta;\Omega)$ is said to $H$-converge to the matrix-valued function $A_*$ if for any distribution $f \in H^{-1}(\Omega)$, the solution $u_\e \in H_0^1(\Omega)$ of the equation $\mathrm{div}(A_\e \nabla u_\e) = f$ satisfies the convergences \eq u_\e \rightharpoonup u_* \text{  weakly in } H_0^1(\Omega) \quad \text{  and  } \quad  A_\e \nabla u_\e \rightharpoonup A_* \nabla u_* \text{ weakly in } L^2(\Omega)^2, \qe where $u_*$ solves in $H_0^1(\Omega)$ the homogenized equation $\mathrm{div}(A_* \nabla u_*) = f$.
\label{Hconv}
\end{Def}

Murat and Tartar \cite{tpeccot} proved that for any sequence $A_\e$ in $\mathcal M(\alpha,\beta;\Omega)$, there exist $A_*$ in $\mathcal M(\alpha,\beta;\Omega)$ and a subsequence of $A_\e$ which $H$-converges to $A_*$.

%We now introduce a classical algebraic tool in order to simplify the expression of tensors from $\R^3$ to $\R_a^{3 \times 3}$:

%\begin{Def}
%Let $\mathscr E$ be the Levi-Civita third-order tensor defined by \eqs \forall \ \xi \in \R^3, \quad \mathscr E(\xi) = \begin{pmatrix} 0 & -\xi_3 & \xi_2 \\ \xi_3 & 0 & -\xi_1 \\ -\xi_2 & \xi_1 & 0 \end{pmatrix}.\qes Then, for any tensor $\mathscr R$ from $\R^3$ to $\R_a^{3 \times 3}$, there exists a unique matrix $R$ such that, \eq \forall \ h \in \R^3, \quad \mathscr R(h) = \mathscr E (R h).\qe
%\label{LC}
%\end{Def}

\section{The three-dimensional effective magneto-resistance}

\label{sect1}

\subsection{The three-dimensional Hall effect and magneto-resistance}

Let $\alpha,\beta >0$, and let $\Omega$ be a regular bounded domain of $\R^3$. Consider a heterogeneous conductor in $\Omega$, with a symmetric matrix-valued conductivity $\sigma_\e \in \mathcal M(\alpha, \beta;\Omega)$ (see \eqref{M}), associated with the resistivity $\rho_\e : = \left(\sigma_\e\right)^{-1}$. Here, $\e$ is a small positive parameter which represents the scale of the microstructure. In the presence of a magnetic field $h \in \R^3$, it is known (see, \textit{e.g.}, \cite{LaLi}) that the perturbed resistivity satisfies the property \eq \rho_\e(-h) = \rho_\e(h)^\T.  \label{cond-sym}\qe Also assume that the conductivity satisfies the following regularity properties: there exist an open ball $O$ in $\R^3$ centered at $0$ and $b \in L^\infty(\Omega)$ such that for any $\e > 0$ and any multi-index $|\nu| \leq 2,$\begin{equation} \; \left\{\!\! \begin{array}{l}
\vspace{0.2cm} \sigma_\e(h) \in \mathcal M(\alpha, \beta;\Omega), \quad \forall \ h \in O, \\
\vspace{0.2cm} h \mapsto \sigma_\e(h)(x) \text{ is of class } \mathscr C^{|\nu|} \text{ on } O, \quad \forall \ x \in \Omega , \\
\displaystyle \left| \frac{\partial^{|\nu|} \sigma_\e}{\partial h^\nu}(h)(x) - \frac{\partial^{|\nu|} \sigma_\e}{\partial h^\nu}(k)(x) \right| \leq b(x) \ |h-k|, \quad \forall \ h,k \in O, \text{ a.e. } x \in \Omega .
\end{array} \right. \;\;\;\label{cond-reg}\end{equation} As a consequence of \eqref{cond-reg}, the resistivity $\rho_\e(h)$ satisfies the second-order expansion\eq \rho_\e(h) = \rho_\e + \mathscr R_\e(h) + \mathscr M_\e(h,h) + o_{L^\infty(\Omega)^{3 \times 3}}(|h|^2),\label{exp-p1}\qe where $\mathscr R_\e:  \R^3 \rightarrow L^\infty(\Omega)^{3 \times 3}$  and $\mathscr M_\e:  \R^3 \times \R^3 \rightarrow L^\infty(\Omega)^{3 \times 3}$ are sequences of linear operators uniformly bounded with respect to $\e$. By virtue of \eqref{cond-sym}, for any $h$ in $\R^3$, $\rho_\e(0)$ and $\mathscr M_\e(h,h)$ are symmetric matrix-valued functions, while $\mathscr R_\e(h)$ is an antisymmetric matrix-valued function. The matrix-valued function defined by (see \eqref{defR} and \cite{BM1} for more details) \eq R_\e h := \mathscr E^{-1} \big( \mathscr R_\e(h) \big), \qe and the second-order term $\mathscr M_\e(h,h)$ are respectively called the Hall matrix and the magneto-resistance associated with the perturbed resistivity $\rho_\e(h)$, so that
\eq \rho_\e(h) = \rho_\e + \mathscr E (R_\e h) + \mathscr M_\e(h,h) + o_{L^\infty(\Omega)^{3 \times 3}}(|h|^2).\label{exp-p} \qe

\begin{Rem}
Since by assumption $\mathscr R_\e :  \R^3 \rightarrow L^\infty(\Omega)^{3 \times 3}$ is uniformly bounded with respect to $\e$, $R_\e$ is a bounded sequence in $L^\infty(\Omega)^{3 \times 3}$.
\end{Rem}

Similarly, we define the linear operators $S_\e :  \R^3 \rightarrow L^\infty(\Omega)^{3 \times 3}$ and $\mathscr N_\e :  \R^3 \times \R^3 \rightarrow L^\infty(\Omega)^{3 \times 3}$ by \eq \sigma_\e(h) = \rho_\e(h)^{-1} = \sigma_\e + \mathscr E (S_\e h) + \mathscr N_\e(h,h) + o_{L^\infty(\Omega)^{3 \times 3}}(|h|^2),\label{exp-s} \qe which are uniformly bounded with respect to $\e$.

\subsection{Homogenization of the magneto-resistance}

Assume that $\sigma_\e(h)$ $H$-converges (see definition \ref{Hconv}) to $\sigma_*(h)$ for any $h \in O$. In fact due to the compactness of $H$-convergence \cite{tpeccot} this holds true for a subsequence of $\e$ and a countable set of~$h$. Then, by \cite{ColSpa} (Theorem 2.5 in the symmetric case) and \cite{BocMur} (Theorem 3.1 in the non symmetric case), the effective (or homogenized) conductivity $\sigma_*(h)$ satisfies $\sigma_*(-h) = \sigma_*(h)^\T$. By the regularity conditions \eqref{cond-reg} (see \cite{Brimag} and \cite{ColSpa} for more details), as in \eqref{exp-s}, we have the second-order expansion\eq \sigma_*(h) = \sigma_* + \mathscr E(S_* h) + \mathscr N_*(h,h) +o(|h|^2).\label{exp-cond*}\qe Moreover, by taking the inverse of \eqref{exp-cond*}, the effective resistivity $\rho_*(h) : =\sigma_*(h)^{-1}$ also expands as \eq \rho_*(h) = \rho_* + \mathscr E(R_* h) + \mathscr M_*(h,h) +o(|h|^2),\label{exp-p*}\qe where $R_* \in L^\infty(\Omega)^{3 \times 3}$ is the effective Hall matrix and $\mathscr M_* :  \R^3 \times \R^3 \rightarrow L^\infty(\Omega)^{3 \times 3}$ is the effective magneto-resistance tensor of the composite. We have the following result:

\begin{Prop}
The following relations hold for any $h \in O$,

\eq S_\e = -\, \mathrm{Cof}(\sigma_\e) \, R_\e \quad \text{and} \quad S_* = -\, \mathrm{Cof}(\sigma_*) \, R_*,  \label{relaS}\qe

\eq \mathscr N_\e(h,h) = - \, \sigma_\e \, \mathscr M_\e(h,h) \, \sigma_\e + \mathscr E(S_\e h) \, \sigma_\e^{-1} \, \mathscr E(S_\e h), \label{relaN}\qe \eq \quad \mathscr N_*(h,h) = - \, \sigma_* \,  \mathscr M_*(h,h) \, \sigma_* + \mathscr E(S_* h) \, \sigma_*^{-1} \, \mathscr E(S_* h). \label{relaN*}\qe
\label{relations}
\end{Prop}

\begin{proof} By the first-order expansions \eqref{exp-p} and \eqref{exp-s} we have, for any $h \in O$ and any $t > 0$ small enough, \begin{align*} 0 &= -I_3 + \big(\sigma_\e + \mathscr E ( t S_\e h) + \mathscr N_\e(t h,t h)\big)\big(\rho_\e + \mathscr E (t R_\e h) + \mathscr M_\e(t h,t h)\big) + o_{L^\infty(\Omega)^{3 \times 3}}(t^2  |h|^2) \\
&= t \big( \mathscr E ( S_\e h) \sigma_\e^{-1} + \sigma_\e \mathscr E ( R_\e h) \big) \\& + t^2 \big(\mathscr N_\e(h,h) \sigma_\e^{-1} + \mathscr E ( S_\e h) \mathscr E ( R_\e h) + \sigma_\e \mathscr M_\e(h,h) \big) + o_{L^\infty(\Omega)^{3 \times 3}}(t^2 |h|^2). \end{align*} Then, dividing by $t$ the previous equality and letting $t$ tend to zero we obtain \eq \mathscr E ( R_\e h) = - \, \sigma_\e^{-1} \, \mathscr E ( S_\e h) \, \sigma_\e^{-1}. \label{eqinter'}\qe Hence, we get that \eq 0 = t^2 \big(\mathscr N_\e(h,h) \ \sigma_\e^{-1} - \mathscr E ( S_\e h) \ \sigma_\e^{-1} \ \mathscr E ( S_\e h) \ \sigma_\e^{-1} + \sigma_\e \ \mathscr M_\e(h,h) \big) + o_{L^\infty(\Omega)^{3 \times 3}}(t^2 |h|^2).\qe We divide this equality by $t^2$ and let $t$ tend to 0 to get \eqref{relaN}.

The first equality of \eqref{relaS} is a straightforward consequence of the following algebraic lemma which is proved in \cite{BM1} (Lemma 1):

\eq \forall \ P \in \R^{3 \times 3}, \quad P^\T \mathscr E(\xi) P = \mathscr E \big( \mathrm{Cof}(P)^\T \xi \big). \label{lemalg}\qe Applying \eqref{lemalg} to the equality \eqref{eqinter}, and using that $\sigma_\e$ is symmetric, we get that for any $h \in O$, \eq \mathscr E ( S_\e h) = -\sigma_\e \mathscr E ( R_\e h) \sigma_\e = - \sigma_\e^\T \mathscr E ( R_\e h) \sigma_\e =   \mathscr E \big( \! -\mathrm{Cof}(\sigma_\e)^\T R_\e h \big) = \mathscr E \big( \! -\mathrm{Cof}(\sigma_\e) R_\e h \big), \qe which shows the first part of \eqref{relaS} due to the invertibility of $\mathscr E$.

The proof for the homogenized quantities in \eqref{relaS} and \eqref{relaN*} is quite similar.
\end{proof}

\subsubsection{The general case}

We now give an analogous in dimension three of Theorem 2.2 of \cite{Brimagneto} in order to give weak convergences of the effective Hall matrix and the magneto-resistance. All the subsequences parametrized by $h$ converge up to a subsequence of $\e$. Due to \eqref{cond-reg}, the linearity or the quadratic dependence in $h$, the convergences hold for any $h$. From now on, we consider a subsequence still denoted by $\e$ such that all the sequences converge as $\e$ tends to $0$ and for any $h$ in $O$.

First of all, we need to introduce a corrector $P_\e(h)$ (or electric field) in the sense of Murat-Tartar (see \cite{tpeccot}), which is the gradient of a vector-valued $U_\e(h)$ associated with the unperturbed conductivity $\sigma_\e$ in $\mathcal M(\alpha,\beta;\Omega)$. To this end consider the solution $U_\e(h)$ in $H^1(\Omega)^3$  of the problem\begin{equation} \; \left\{\!\! \begin{array}{r c l l}
\vspace{0.2cm}\Div \big(\sigma_\e(h) DU_\e(h) \big) &=& \Div\big(\sigma_*(h)\big) \ \ &\mathrm{in} \ \mathscr D'(\Omega), \\
\displaystyle U_\e(h)(x) - x &=& 0 &\text{on } \partial \Omega.
\end{array} \right. \;\;\;\label{def-corr}\end{equation} Thanks to $H$-convergence and the Meyers estimate of \cite{Mey}, there exists a number $p > 2$ which only depends on $\alpha, \beta, \Omega$, such that the corrector $P_\e(h) :=DU_\e(h)$ satisfies, for any $h \in O$,\eq P_\e(h) \rightharpoonup I_3 \quad \text{ weakly in }L^p(\Omega)^{3 \times 3}. \label{Prop-P}\qe The knowledge of such a corrector combined with the div-curl lemma (see \cite{MurComp} and \cite{TarComp}) permits to derive the effective perturbed effective conductivity by the following convergence \eq \sigma_\e(h) P_\e(h) \rightharpoonup \sigma_*(h) \quad \text{ weakly in }L^p(\Omega)^{3 \times 3}. \label{conv-P}\qe By the regularity condition \eqref{cond-reg}, the coercivity of $\sigma_\e(h)$ and the Meyers estimate \cite{Mey}, the potential $U_\e(h)$ and the corrector $P_\e(h)$ admit the following second-order expansions in $h$:\begin{align} \vspace{0.4cm} U_\e(h) = \vspace{2cm} U_\e^0 + U_\e^1(h) + U_\e^2(h,h) + o_{W^{1,p}(\Omega)^3}(|h|^2)\label{exp-U}, \\ \vspace{2cm} P_\e(h) = P_\e^0 + P_\e^1(h) + P_\e^2(h,h) + o_{L^p(\Omega)^{3 \times 3}}(|h|^2).\label{exp-P}\end{align} We can now state the following result:

\begin{Theo}
Assume that the conditions \eqref{cond-sym}-\eqref{exp-p*} are satisfied, and that the norms of the Hall matrix $R_\e$ and the local magneto-resistance tensor $\mathscr M_\e$ are bounded in $L^\infty(\Omega)$. Then, the effective Hall matrix $R_*$, the effective $S$-matrix $S_*$ and the effective magneto-resistance are given by the following limits, for any $h \in O$, \eq S_* = \lim \limits_{w-L^1(\Omega)}\mathrm{Cof}(P_\e^0)^\T \, S_\e, \quad \mathrm{Cof}(\sigma_*) \, R_* = \lim \limits_{w-L^1(\Omega)}\mathrm{Cof}(\sigma_\e P_\e^0)^\T \, R_\e,  \label{Thconv}\qe and \eq\begin{array}{ r c l}
   \sigma_* \, \mathscr M_*(h,h) \, \sigma_* & = &\lim \limits_{w-L^1(\Omega)} \big( \sigma_\e P_\e^0\big)^\text{T} \, \mathscr M_\e(h,h) \, \big( \sigma_\e P_\e^0\big) - \mathscr E(S_* h)^\text{T} \, \sigma_*^{-1} \, \mathscr E(S_* h) \\
   & + & \lim \limits_{w-L^1(\Omega)} \big( \mathscr E(S_\e h) P_\e^0 + \sigma_\e P_\e^1(h)\big)^\text{T} \, \sigma_\e^{-1} \, \big( \mathscr E(S_\e h) P_\e^0 + \sigma_\e P_\e^1(h)\big),
\end{array} \label{Thconv1}\qe where $w - L^1(\Omega)$ means that the convergence holds weakly in $L^1(\Omega)$  and $P_\e^0$, $P_\e^1(h)$ are the matrix-valued gradient which satisfy \eqref{exp-P}. \label{limites} \label{THconv}
\end{Theo}

\begin{proof} The proof uses similar expansions as in \cite{Brimag} combined with algebraic specificities of dimension~3. Taking into account the expansions \eqref{exp-s} and \eqref{exp-P}, we have: \begin{eqnarray}
   \sigma_\e(h) P_\e(h) & = & \sigma_\e P_\e^0 + \big( \sigma_\e P_\e^1(h) + \mathscr E(S_\e h) P_\e^0 \big) \nonumber \\
   & + &\big(\sigma_\e P_\e^2(h,h) + \mathscr E(S_\e h) P_\e^1(h) + \mathscr N_\e(h,h) P_\e^0 \big) + o_{L^2(\Omega)^{3 \times 3}}(|h|^2). \label{exp-sP}
\end{eqnarray} By virtue of Remark \ref{Remconv}, using the properties \eqref{def-corr}-\eqref{conv-P} satisfied by the corrector $P_\e(h)$ in the expansions \eqref{exp-P}, \eqref{exp-sP} and \eqref{exp-cond*}, we get that\begin{equation} \; \left\{\!\! \begin{array}{r c l}
\vspace{0.2cm} P_\e^0 & \harpoon & I_3 \quad \text{weakly in } L^{p}(\Omega)^{3 \times 3},\\
\vspace{0.2cm} P_\e^1(h) & \harpoon & 0  \ \quad \text{weakly in } L^{p}(\Omega)^{3 \times 3},\\
P_\e^2(h,h) & \harpoon & 0 \ \quad \text{weakly in } L^{p}(\Omega)^{3 \times 3},
\end{array} \right. \;\;\;\label{conv3}\end{equation} and \begin{equation} \; \left\{\!\! \begin{array}{r c l}
 \vspace{0.2cm}\Div\big(\sigma_\e P_\e^0\big) & = & \Div\big(\sigma_*\big),\\
\Div\big(\sigma_\e P_\e^1(h) + \mathscr E(S_\e h) P_\e^0 \big) & = & \Div\big(\mathscr E(S_* h)\big) 
\end{array} \quad \text{ in } \mathscr D'(\Omega)^{3 \times 3}. \right. \label{div2}\;\;\;\end{equation}Moreover, from the expansions \eqref{exp-P}, \eqref{exp-sP} and the symmetry of $\sigma_\e$, we deduce that\begin{eqnarray}
   P_\e(h)^\T\sigma_\e(h) P_\e(h) & = & \big(P_\e^0)^\T \sigma_\e P_\e^0 + \big(P_\e^0)^\T \mathscr E(S_\e h) P_\e^0 + \big(P_\e^0)^\T \big( \mathscr E(S_\e h) P_\e^1(h) + \mathscr N_\e(h,h) P_\e^0 \big) \nonumber \\
   & + &\big(\sigma_\e P_\e^0)^\T  P_\e^2(h,h) + \big(\sigma_\e P_\e^0)^\T  P_\e^1(h) \nonumber \\
   &  + &\big(P_\e^1(h)\big)^\T \sigma_\e(h) P_\e(h) + \big(P_\e^2(h,h)\big)^\T \sigma_\e(h) P_\e(h)  + o_{L^{p/2}(\Omega)^{3 \times 3}}(|h|^2). \label{exp-PsP}
\end{eqnarray} Then, taking into account \eqref{Prop-P}, \eqref{conv3}, \eqref{div2}, the div-curl lemma implies that $P_\e(h)^\T\sigma_\e(h) P_\e(h)$ converges to $\sigma_*(h)$ in $L^{p/2}(\Omega)^{3 \times 3}$, and $\big(P_\e^1(h)\big)^\T \sigma_\e(h) P_\e(h)$, $\big(P_\e^2(h,h)\big)^\T \sigma_\e(h) P_\e(h)$, $\big(\sigma_\e P_\e^0)^\T  P_\e^2(h,h)$, $\big(\sigma_\e P_\e^0)^\T  P_\e^1(h)$ converges to $0$ in $L^{p/2}(\Omega)^{3 \times 3}$. Noting that $\big(P_\e^0\big)^\T \mathscr E(S_\e h) P_\e^0 = \mathscr E \big(\text{Cof}(P_\e^0)^\T S_\e h\big)$ by \eqref{lemalg} and passing to the limit in \eqref{exp-PsP}, we obtain \begin{eqnarray}
   \sigma_*(h) & = & \sigma_* +  \lim \limits_{w-L^1(\Omega)} \mathscr E \big(\text{Cof}(P_\e^0)^\T S_\e h\big) \nonumber \\
   & + &\lim \limits_{w-L^1(\Omega)} \Big[ \big(P_\e^0\big)^\T \mathscr E(S_\e h) P_\e^1(h) + \big(P_\e^0\big)^\T \mathscr N_\e(h,h) P_\e^0 \Big]  +o_{L^{1}(\Omega)^{3 \times 3}}(|h|^2).\label{lim-PsP}
\end{eqnarray} Equating this expression with \eqref{exp-cond*} it follows that \eq \label{convS}\mathscr E \big( S_* h\big) = \lim \limits_{w-L^1(\Omega)} \mathscr E \big(\text{Cof}(P_\e^0)^\T S_\e h\big), \qe and \eq \mathscr N_*(h,h) = \lim \limits_{w-L^1(\Omega)} \Big[\big(P_\e^0\big)^\T \mathscr E(S_\e h) P_\e^1(h) + \big(P_\e^0\big)^\T \mathscr N_\e(h,h) P_\e^0\Big].\label{Ninter}\qe As $\mathscr E$ is an invertible linear mapping, we deduce from \eqref{convS} and \eqref{relaS} the convergences \eqref{Thconv}.

We have, as $\mathscr E(S_\e h)$ is antisymmetric and $\sigma_\e$ symmetric, \begin{align}
   &\hspace{-0.30cm}\big(\sigma_\e P_\e^1(h) + \mathscr E(S_\e h) P_\e^0 \big)^\T \big( \sigma_\e \big)^{-1} \big(\sigma_\e P_\e^1(h) + \mathscr E(S_\e h) P_\e^0 \big) \nonumber \\   
   &\hspace{-0.30cm} = - \big(P_\e^0 \big)^\T \mathscr E(S_\e h) \big( \sigma_\e \big)^{-1} \mathscr E(S_\e h)P_\e^0 - \big(P_\e^0 \big)^\T \mathscr E(S_\e h) P_\e^1(h) + \big(P_\e^1(h)\big)^\T \big(\sigma_\e P_\e^1(h) + \mathscr E(S_\e h) P_\e^0 \big). \label{eqinter}
\end{align} Again, taking into account \eqref{Prop-P}, \eqref{conv3}, \eqref{div2}, the div-curl lemma implies that \eq \big(P_\e^1(h)\big)^\T \big(\sigma_\e P_\e^1(h) + \mathscr E(S_\e h) P_\e^0 \big) \harpoon 0 \quad \text{ weakly in }L^p(\Omega)^{3 \times 3}. \qe Hence, \eqref{eqinter} implies that\begin{align}
   & \lim \limits_{w-L^1(\Omega)} \big(\sigma_\e P_\e^1(h) + \mathscr E(S_\e h) P_\e^0 \big)^\T  \sigma_\e^{-1} \big(\sigma_\e P_\e^1(h) + \mathscr E(S_\e h) P_\e^0 \big) \nonumber \\ & =  - \lim \limits_{w-L^1(\Omega)} \Big[ \big(P_\e^0 \big)^\T \mathscr E(S_\e h) \sigma_\e^{-1} \mathscr E(S_\e h)P_\e^0 + \big(P_\e^0 \big)^\T \mathscr E(S_\e h) P_\e^1(h)\Big].  \label{convinter}
\end{align} Combining the equalities \eqref{relaN}, \eqref{relaN*} of Proposition \ref{relations} with \eqref{Ninter}, \eqref{convinter} and the antisymmetry of $\mathscr E(S_\e h)$ and $\mathscr E(S_* h)$, we obtain that \begin{align}
&\sigma_* \, \mathscr M_*(h,h) \, \sigma_*  - \lim \limits_{w-L^1(\Omega)} \big( \sigma_\e P_\e^0\big)^\text{T} \mathscr M_\e(h,h) \big( \sigma_\e P_\e^0\big) \nonumber\\
&=\lim \limits_{w-L^1(\Omega)} \Big[ \big(P_\e^0\big)^\T \mathscr N_\e(h,h)P_\e^0 - \big(P_\e^0\big)^\T \mathscr E(S_\e h) \sigma_\e^{-1} \mathscr E(S_\e h)P_\e^0 \Big]\nonumber - \mathscr N_*(h,h) + \mathscr E(S_* h) \ \sigma_*^{-1} \mathscr E(S_* h) \nonumber\\
& = - \lim \limits_{w-L^1(\Omega)} \Big[\big(P_\e^0\big)^\T \mathscr E(S_\e h) \sigma_\e^{-1} \mathscr E(S_\e h)P_\e^0 + \big(P_\e^0\big)^\T \mathscr E(S_\e h) P_\e^1(h) \Big] \nonumber - \mathscr E(S_* h)^\T \sigma_*^{-1} \mathscr E(S_* h) \nonumber\\
& = \lim \limits_{w-L^1(\Omega)} \Big[\big(\sigma_\e P_\e^1(h) + \mathscr E(S_\e h) P_\e^0 \big)^\T \big( \sigma_\e \big)^{-1} \big(\sigma_\e P_\e^1(h) + \mathscr E(S_\e h) P_\e^0 \big) \Big] \nonumber - \mathscr E(S_* h)^\T \sigma_*^{-1} \mathscr E(S_* h).
\end{align} In fact, the convergences \eqref{Thconv} and \eqref{Thconv1} hold in $L^{p/2}(\Omega)^{3 \times 3}$. \end{proof}

\subsubsection{The periodic case}

We now give a corollary of Theorem \ref{limites} for periodic media. To this end, set $Y:=(0,1)^3$, and consider the $\varepsilon Y$-periodic conductivity \eq \sigma_\e(h)(x):=\sigma(h)\left(\frac{x}{\e}\right),\qe where $\sigma(h)$ is a $Y$-periodic matrix-valued function. We assume \eqref{cond-sym} and analogous regularity conditions to \eqref{cond-reg}: there exists an open ball $O$ in $\R^3$ centered at $0$ such that \begin{equation} \; \left\{\!\! \begin{array}{l}
\vspace{0.2cm} \sigma(h) \in \mathcal M(\alpha, \beta;\Omega), \quad \forall \ h \in O, \\ h \mapsto \sigma(h)(y) \text{ is of class } \mathscr C^2 \text{ on } O, \quad \forall \ y \in Y.
\end{array} \right. \;\;\;\label{cond-regper}\end{equation} These conditions gives the expansions, like in \eqref{exp-s} and \eqref{exp-p1}  \begin{eqnarray} \sigma(h) = \sigma + \mathscr E (S h) + \mathscr N(h,h) + o_{L^\infty(\Omega)^{3 \times 3}}(|h|^2),\label{exp-condper}\\ \rho(h) = \sigma(h)^{-1}= \rho + \mathscr E (R h) + \mathscr M(h,h) + o_{L^\infty(\Omega)^{3 \times 3}}(|h|^2),\label{exp-p1per}\end{eqnarray} where $\sigma$, $\rho$, $S$, $R$, $\mathscr N$ and $\mathscr M$ are $Y$-periodic functions bounded in $Y$. The corrector $P_\e(h):=DU_\e(h)$ defined in \eqref{def-corr} and \eqref{exp-U} reads as $U_\e(h):=\e U(h)\left(\frac{x}{\e}\right)$, where $U(h)$ is the unique solution in $H_{\text{loc}}^1(\R^3)$ (up to an additive constant) of the problem\begin{equation} \; \left\{\!\! \begin{array}{ r c l l}
\vspace{0.2cm} \Div \big(\sigma(h) DU(h) \big) & =  & 0 & \text{in } \mathscr D'(\R^3), \\
\displaystyle y \mapsto U(h)(y) - y & = & 0 & \text{is $Y$-periodic.}
\end{array} \right. \;\;\;\label{def-corrper}\end{equation} and $P(h):=DU(h)$ with \eq\langle P(h)\rangle = I_3.\label{condmoyper}\qe We have the classical periodic homogenization formula (see, \textit{e.g.}, \cite{tpeccot} for more details) \linebreak[4] $\sigma_*(h) =\left \langle \sigma(h) DU(h) \right \rangle$. By virtue of \eqref{cond-regper} we have expansions similar to \eqref{exp-cond*}, \eqref{exp-p*} and \eqref{exp-P} \eq \sigma_*(h) = \sigma_* + \mathscr E(S_* h) + \mathscr N_*(h,h) +o(|h|^2),\label{exp-cond*per}\qe\eq \rho_*(h) = \sigma_*(h)^{-1} = \rho_* + \mathscr E(R_* h) + \mathscr M_*(h,h) +o(|h|^2),\label{exp-p*per}\qe and\eq P(h) = P^0 + P^1(h) + P^2(h,h) + o_{L^2(\Omega)^{3 \times 3}}(|h|^2).\label{exp-Pper}\qe We can state a corollary to Theorem \ref{limites}:

\begin{Cor}
For a periodic conductor, the effective Hall matrix $R_*$, the effective $S$-matrix $S_*$  and the effective magneto-resistance tensor $\mathscr M_*$ are given by the following relations, for any $h \in O$, \eq S_* = \big \langle \mathrm{Cof}(P^0)^\T S \big \rangle, \quad \mathrm{Cof}(\sigma_*) \ R_* = \big \langle\mathrm{Cof}(\sigma_\e P^0)^\T R \big \rangle, \label{Thconvper}\qe and\begin{eqnarray}
   \sigma_* \, \mathscr M_*(h,h) \, \sigma_* & = &\Big \langle \big( \sigma P^0\big)^\text{T} \mathscr M(h,h) \big( \sigma P^0\big) \Big \rangle - \mathscr E(S_* h)^\text{T} \sigma_*^{-1} \mathscr E(S_* h) \nonumber \\
   & + & \Big \langle \big( \mathscr E(S h) P^0 + \sigma P^1(h)\big)^\text{T} \sigma^{-1} \big( \mathscr E(S h) P^0 + \sigma P^1(h)\big) \Big \rangle,
\label{Thconv1per}\end{eqnarray}\label{limitesper}where $P^0$, $P^1(h)$ are the matrix-valued gradient which satisfy \eqref{exp-Pper}. 
\label{THconvper}
\end{Cor}
\label{secper}

\section{Comparison between the effective magneto-resistance and the local magneto-resistance.}
\label{sect2}
\subsection{The main result}

We now give a generalization of the two-dimensional Theorem 2.4 of \cite{Brimagneto}, and a corollary in the periodic case with the notations of Section \ref{secper}.

\begin{Theo}
Assume that the conditions \eqref{cond-sym}-\eqref{exp-p*} are satisfied, and that the norm of the local Hall matrix $R_\e$ and the norm of the local magneto-resistance tensor $\mathscr M_\e$ are bounded in $L^\infty(\Omega)$. Then, for any $h \in O$, we have \eq
   \sigma_* \, \mathscr M_*(h,h) \, \sigma_* \geq \lim \limits_{w-L^1(\Omega)} \big( \sigma_\e P_\e^0\big)^\text{T} \mathscr M_\e(h,h) \big( \sigma_\e P_\e^0\big). \label{ineq}\qe Moreover, \eqref{ineq} is an equality if and only if \eq \mathrm{Curl} \left( \mathscr{E}(R_\e h) \sigma_\e P_\e^0 \right) \text{ lies in a compact subset of } H^{-1}(\Omega)^{3 \times 3 \times 3}. \label{condRot}\qe \label{Thineq}
\end{Theo}\begin{Cor}
In the periodic case, the constant effective magneto-resistance tensor $\mathscr M_*$ and the constant effective conductivity $\sigma_*$ satisfy the inequality for any $h \in O$, \eq
   \sigma_* \, \mathscr M_*(h,h) \, \sigma_* \geq \Big\langle \big( \sigma P^0\big)^\text{T} \mathscr M(h,h) \big( \sigma P^0\big) \Big \rangle , \quad \text{with } \sigma_* = \big \langle \sigma P^0 \big \rangle , \label{ineqper}\qe where $\sigma(y)$ is the local conductivity and $\mathscr M(h,h)(y)$ is the local magneto-resistance. Moreover, \eqref{ineqper} is an equality if and only if \eq \mathrm{Curl} \left( \mathscr{E}(R h) \ \! \sigma \ \! P^0 \right) =0 \quad \text{in } \mathscr D'(\R^3)^{3 \times 3 \times 3}. \label{condRotper}\qe \label{Thineqper}\end{Cor}
   
\begin{Rem}
 Then the inequality \eqref{ineqper} can be written, for any $h \in \R^3$, \eq \mathscr M_*(h,h)\langle j \rangle \cdot \langle j \rangle \geq \big \langle \mathscr M(h,h)j \cdot j \big \rangle, \quad \text{with } \langle j \rangle = \sigma_*\langle e \rangle , \label{new}\qe where $e(y) = P^0(y) \langle e \rangle$ is the local electric field and $j(y) = \sigma(y) e(y)$ is the local current field. Inequality \eqref{new} means that the dissipation energy in a composite is greater than or equal to the average of the dissipation energy in each of its phases.
\label{electric}
\end{Rem} 
   
 The proof of Theorem \ref{Thineq} is based on the following result:

\begin{Lem}
Let $\Omega$ be a bounded open subset of $\R^d$. Consider a sequence $A_\e$ of symmetric matrix-valued functions in $\mathcal M(\alpha,\beta;\Omega)$ which $H$-converges to $A_*$, and a sequence $\xi_\e$ of $L^2(\Omega)^{d}$ which satisfies\eq \xi_\e \harpoon \xi \text{ weakly in } L^2(\Omega)^{d } \quad \text{and} \quad \mathrm{div} \left( \xi_\e \right) \longrightarrow \mathrm{div} \left( \xi \right) \text{ strongly in } H^{-1}(\Omega) \label{prop-xi} .\qe Also assume that \eq A_\e^{-1} \xi_\e \cdot \xi_\e \harpoon \zeta \quad \text{ weakly-$*$ in } \mathcal M(\Omega). \label{lem-ineq'}\qe Then, we have the inequality \eq \zeta \geq A_*^{-1} \xi \cdot \xi \quad \text{in } \mathcal M(\Omega)\label{lem-ineq}.\qe Moreover, the inequality \eqref{lem-ineq} is an equality if and only if \eq \mathrm{curl} \big( A_\e^{-1} \xi_\e \big) \text{ lies in a compact subset of } H_{\mathrm{loc}}^{-1}(\Omega)^{d \times d}.\label{rotcompact}\qe
\label{divcurl}
\end{Lem}

\begin{Rem}
Inequality \eqref{lem-ineq} is a classical duality inequality in the periodic case (see \cite{Jikov} pp.160--200). However, up our knowledge the non-periodic case and the condition \eqref{rotcompact} of equality are less classical and deserve a proof.
\end{Rem}

\noindent \textit{Proof of Theorem \ref{Thineq}.}  Taking into account the expansions \eqref{exp-s} and \eqref{exp-P}, we have as in \eqref{exp-sP}: \begin{eqnarray}
   \sigma_\e(h) P_\e(h) & = & \sigma_\e P_\e^0 + \big( \sigma_\e P_\e^1(h) + \mathscr E(S_\e h) P_\e^0 \big) + o_{L^p(\Omega)^{3 \times 3}}(h). \label{exp-sP'}
\end{eqnarray} By virtue of Remark \ref{Remconv}, using the properties \eqref{def-corr}-\eqref{conv-P} satisfied by the corrector $P_\e(h)$ in the expansions \eqref{exp-P}, \eqref{exp-sP} and \eqref{exp-cond*}, we have, like in \eqref{conv3} and \eqref{div2}:\begin{equation} \; \left\{\!\! \begin{array}{r c l}
   \vspace{0.2cm}P_\e^0 & \harpoon & I_3 \quad \quad \quad \ \ \text{weakly in } L^{p}(\Omega)^{3 \times 3},\\
\vspace{0.2cm}P_\e^1(h) & \harpoon & 0 \quad \quad \quad \ \ \ \text{weakly in } L^{p}(\Omega)^{3 \times 3},\\
\sigma_\e P_\e^1(h) + \mathscr E(S_\e h) P_\e^0 & \harpoon & \mathscr E(S_*h) \quad \ \text{weakly in } L^{p}(\Omega)^{3 \times 3}.
\end{array}  \right. \;\;\;\label{conv3'}\end{equation}
 
 Let $\lambda \in \R^3$. We apply Lemma \ref{divcurl} with $A_\e:=\sigma_\e$, $\xi_\e:=\big(\sigma_\e P_\e^1(h) + \mathscr E(S_\e h) P_\e^0 \big) \lambda$, $\xi:= \mathscr E(S_*h)\lambda$ and \eq \zeta:= \lim \limits_{w-L^1(\Omega)} \big( \mathscr E(S_\e h) P_\e^0 + \sigma_\e P_\e^1(h)\big)^\text{T} \sigma_\e^{-1} \big( \mathscr E(S_\e h) P_\e^0 + \sigma_\e P_\e^1(h)\big)\lambda \cdot \lambda. \qe It follows that for any $\lambda \in \R^3$, \eq \lim \limits_{w-L^1(\Omega)} \big( \mathscr E(S_\e h) P_\e^0 + \sigma_\e P_\e^1(h)\big)^\text{T} \sigma_\e^{-1} \big( \mathscr E(S_\e h) P_\e^0 + \sigma_\e P_\e^1(h)\big) \lambda \cdot \lambda \geq \mathscr E(S_* h)^\text{T} \sigma_*^{-1} \mathscr E(S_* h) \lambda \cdot \lambda. \label{tty1'}\qe Using the fact that $P_\e^1(h)$ is a gradient and \eqref{eqinter'}, \eqref{tty1'} is an equality if and only if \eq \mathrm{Curl} \ \! \big( A_\e^{-1} \xi_\e \big) = \mathrm{Curl} \ \! \big( \sigma_\e^{-1} \mathscr E(S_\e h)  P_\e^0 \lambda \big) = - \, \mathrm{Curl} \ \! \big( \mathscr E(R_\e h) \sigma_\e P_\e^0 \lambda\big) \label{tty2'}\qe lies in a compact subset of $H_{\text{loc}}^{-1}(\Omega)^{3\times 3}$. Due to the arbitrariness of $\lambda$, this can be rewritten \eq \lim \limits_{w-L^1(\Omega)} \big( \mathscr E(S_\e h) P_\e^0 + \sigma_\e P_\e^1(h)\big)^\text{T} \sigma_\e^{-1} \big( \mathscr E(S_\e h) P_\e^0 + \sigma_\e P_\e^1(h)\big) \geq \mathscr E(S_* h)^\text{T} \sigma_*^{-1} \mathscr E(S_* h), \label{tty1}\qe which is an equality if and only if\eq \mathrm{Curl} \ \! \big( \mathscr E(R_\e h) \sigma_\e P_\e^0 \big) \label{tty2} \quad \text{lies in a compact subset of } H_{\text{loc}}^{-1}(\Omega)^{3\times 3 \times 3}. \qe We conclude to \eqref{ineq} by \eqref{Thconv1} and \eqref{tty1}, and to \eqref{condRot} by \eqref{tty2}. \qed

\bigskip

\noindent \textit{Proof of Lemma \ref{divcurl}.} 

\smallskip

\noindent \textit{Proof of inequality \eqref{lem-ineq}:} Let $\varphi$ be a non-negative function in $\mathscr C_c^0(\Omega)$. Let $\delta >0$, and for $i=1,\ldots,k$, let $\lambda_i\in \R^d$ and let $\omega_i$ be balls in $\Omega$ such that \eq \supp \varphi \subset \bigcup_{i = 1}^k \omega_i \quad \text{and} \quad \Sum_{i = 1}^k \int_{\omega_i} |A_*^{-1}\xi - \lambda_i|^2 \dx \leq \delta. \label{approxxi}\qe  We consider a partition of unity $(\psi_i)_{1\leq i \leq k}$ such that \eq \forall \ i=1,\ldots,k, \quad \psi_i \in \mathscr C_c^\infty(\omega_i), \quad 1\geq \psi_i \geq 0, \quad \Sum_{i = 1}^k \psi_i \equiv 1 \text{ in } \supp \varphi, \label{defpartunity}\qe and a sequence of functions $(\widetilde{\psi}_i)_{1\leq i \leq k}$ such that \eq \forall \ i=1,\ldots,k, \quad \widetilde{\psi}_i \in \mathscr C_c^\infty(\omega_i), \quad \widetilde{\psi}_i \equiv 1 \text{ in } \supp \psi_i.\label{psitilde}\qe For $i=1,\ldots,k$, let $ v_\e^i$ be the unique solution of the problem\begin{equation} \; \left\{\!\! \begin{array}{r c l l}
\vspace{0.2cm}\div \big(A_\e \nabla v_\e^i \big) & = & \div \big(A_* \nabla ( \widetilde{\psi}_i \lambda_i \cdot x) \big)  &\text{in } \mathscr D'(\Omega) \\
\displaystyle v_\e^i  & = & 0 &\text{on } \partial \Omega.
\end{array} \right. \;\;\;\label{def-eta}\end{equation} Thanks to the $H$-convergence of $A_\e$ (see Definition \ref{Hconv}) we have the convergence\begin{equation} \; \forall \ i=1,\ldots,k, \quad \left\{\!\! \begin{array}{r c l l}
\vspace{0.2cm}v_\e^i & \harpoon  & v^i = \widetilde{\psi}_i \lambda_i \cdot x \quad &\text{weakly in } H_0^1(\Omega), \\
\nabla v^i  & \equiv & \lambda_i \quad & \text{in } \supp \psi_i.
\end{array} \right. \;\;\;\label{convfuncinter}\end{equation} More generally, for any $\lambda \in \R^d$, we consider the unique solution $ v_\e^\lambda$ of the problem\begin{equation} \; \left\{\!\! \begin{array}{r c l l}
\vspace{0.2cm}\div \big(A_\e \nabla v_\e \big)  & = & \div \big(A_*\lambda \big) & \text{in } \mathscr D'(\Omega), \\
\displaystyle v_\e & = & \lambda \cdot x & \text{on } \partial \Omega.
\end{array} \right. \;\;\;\label{def-eta'}\end{equation} Again, by the $H$-convergence of $A_\e$, we have the convergences\begin{equation} \; \quad \left\{\!\! \begin{array}{r c l l}
\vspace{0.2cm}v_\e & \harpoon &\lambda \cdot x \quad &\text{weakly in } H_0^1(\Omega), \\
A_\e \nabla v_\e &\harpoon &A_* \lambda \quad &\text{weakly in } L^2(\Omega)^d.
\end{array} \right. \;\;\;\label{convfuncinter2}
\end{equation}

We have by \eqref{defpartunity} \eq \int_\Omega \zeta \varphi \dx = \lim \limits_{\e \to 0} \int_\Omega A_\e^{-1} \xi_\e \cdot \xi_\e \ \varphi \dx = \Sum_{i=1}^k \lim \limits_{\e \to 0} \int_{\omega_i} A_\e^{-1} \xi_\e \cdot \xi_\e   \ \psi_i \ \varphi \dx. \qe Combining this inequality with, for $i=1,\ldots,k$,  \eq A_\e^{-1} \xi_\e \cdot \xi_\e - 2\xi_\e \cdot \nabla v_\e^i + A_\e \nabla v_\e^i \cdot \nabla v_\e^i = A_\e^{-1} \big( \xi_\e - A_\e \nabla v_\e^i \big) \cdot \big( \xi_\e - A_\e \nabla v_\e^i \big) \geq 0, \qe we obtain that \eq \int_\Omega \zeta \varphi \dx  \geq \Sum_{i=1}^k \liminf \limits_{\e \to 0} \int_{\omega_i} \big( 2\xi_\e \cdot \nabla v_\e^i - A_\e \nabla v_\e^i \cdot \nabla v_\e^i \big)   \psi_i \varphi \dx. \label{ineqapprox}\qe

By \eqref{def-eta}, \eqref{prop-xi} and \eqref{convfuncinter}, and by the classical div-curl lemma of \cite{MurComp,tpeccot} we have \eq \xi_\e \cdot \nabla v_\e^i \harpoon \xi \cdot \nabla v^i = \xi \cdot \lambda_i \quad \text{and} \quad A_\e \nabla v_\e^i \cdot \nabla v_\e^i\harpoon A_* \nabla v^i \cdot \nabla v^i = A_* \lambda_i \cdot \lambda_i\label{div-curlapp}\qe weakly-$*$ in $\mathcal M(\omega_i)$. This combined with \eqref{ineqapprox} and \eqref{defpartunity} yields \begin{align} \int_\Omega \zeta \varphi \dx  & \geq \Sum_{i=1}^k \int_{\omega_i} \big( 2\xi \cdot \lambda_i - A_* \lambda_i \cdot \lambda_i \big)   \ \psi_i \ \varphi \dx \nonumber \\ &= \Sum_{i=1}^k \int_{\omega_i} A_*^{-1} \xi \cdot \xi \ \psi_i \ \varphi \dx - \Sum_{i=1}^k \int_{\omega_i} A_* \big( A_*^{-1} \xi -  \lambda_i \big) \cdot \big( A_*^{-1}\xi - \lambda_i \big) \ \psi_i \ \varphi \dx \nonumber \\ &= \int_{\Omega} A_*^{-1} \xi \cdot \xi \ \varphi \dx - \Sum_{i=1}^k \int_{\omega_i} A_* \big( A_*^{-1} \xi -  \lambda_i \big) \cdot \big( A_*^{-1}\xi - \lambda_i \big) \  \psi_i \ \varphi \dx. \label{grosseineq}\end{align}  Moreover, by the Cauchy-Schwarz inequality and $A_* \in \mathcal M(\alpha,\beta;\Omega)$, we have for $i=1,\ldots,k$, \eq \int_{\omega_i} A_* \big( A_*^{-1} \xi -  \lambda_i \big) \cdot \big( A_*^{-1}\xi - \lambda_i \big) \ \psi_i \ \varphi \dx \leq \beta \ \Vert \varphi \Vert_\infty \int_{\omega_i}  |A_*^{-1}\xi - \lambda_i|^2 \dx.\qe Summing these inequalities on $i$ together with \eqref{grosseineq} and \eqref{approxxi}, we finally get that \eq \int_\Omega \zeta \varphi \dx \geq  \int_{\Omega} A_*^{-1} \xi \cdot \xi \varphi \dx -\beta \ \delta \ \Vert \varphi \Vert_\infty.\qe We conclude to \eqref{ineq} since $\delta$ is arbitrary.

\bigskip

\noindent \textit{Proof of the case of equality:} Let us now prove that the equality in \eqref{lem-ineq} implies \eqref{rotcompact}. Consider a compact subset $K$ of $\Omega$, and a sequence $\Phi_\e \in H_0^1(\Omega)^{d \times d}$ such that\eq \Phi_\e \harpoon 0 \text{ weakly in } H_0^1(\Omega)^{d \times d} \quad \text{and} \quad \supp \Phi_\e \subset K \label{defPhi}.\qe By the definitions of the curl and the divergence, and by \eqref{defPhi} we have \eq \big\langle \curl\big(A_\e^{-1} \xi_\e \big), \Phi_\e \big\rangle_{H^{-1}(\Omega)^{d \times d}, H_0^1(\Omega)^{d \times d}} = \int_K  A_\e^{-1} \xi_\e \cdot \Div\big( \Phi_\e - \Phi_\e^\T \big) \dx. \label{intpart1}\qe

Consider a partition of unity $(\psi_i)_{1\leq i \leq k}$ such that \eq \forall \ i=1,\ldots,k, \quad \psi_i \in \mathscr C_c^\infty(\omega_i), \quad 0\leq \psi_i \leq 1, \quad \Sum_{i = 1}^k \psi_i \equiv 1 \text{ in } K, \label{defpartunity1}\qe the functions $\widetilde{\psi}_i$ defined by \eqref{psitilde}, and the function $v_\e^i$ by \eqref{def-eta}. We decompose the equality \eqref{intpart1} in two parts \eq \big \langle \curl\big(A_\e^{-1} \xi_\e \big), \Phi_\e \big \rangle_{H^{-1}(\Omega)^{d \times d}, H_0^1(\Omega)^{d \times d}} = I_\e + J_\e,\label{decomp}\qe where \eq I_\e := \Sum_{i = 1}^k \int_{\omega_i} \big( A_\e^{-1} \xi_\e - \nabla v_\e^i\big) \cdot \Div\big( \Phi_\e - \Phi_\e^\T \big) \ \psi_i \dx, \qe and \eq J_\e := \Sum_{i = 1}^k \int_{\omega_i} \nabla v_\e^i \cdot \Div\big( \Phi_\e - \Phi_\e^\T \big) \ \psi_i \dx.\qe 

On the one hand, by the Cauchy-Schwarz inequality we have \eq |I_\e|^2 \leq \left(\Sum_{i = 1}^k \int_{\omega_i} |A_\e^{-1}\xi_\e - \nabla v_\e^i|^2 \psi_i \dx\right) \! \left(\Sum_{i = 1}^k \int_{\omega_i} \left| \Div\big( \Phi_\e - \Phi_\e^\T \big)  \right|^2 \psi_i\dx \right),\qe that is \eq |I_\e|^2 \leq \big \Vert \Div\big( \Phi_\e - \Phi_\e^\T \big) \big \Vert^2_{L^2(\Omega)^{d}} \ \Sum_{i = 1}^k \int_{\omega_i} |A_\e^{-1}\xi_\e - \nabla v_\e^i|^2 \ \psi_i \dx. \label{I}\qe  Using successively \eqref{lem-ineq}, \eqref{div-curlapp} and \eqref{lem-ineq'}, we get that \begin{align}
 &\hspace{-0.45cm}\limsup \limits_{\e \to 0}\Sum_{i = 1}^k \int_{\omega_i} \big |A_\e^{-1}\xi_\e - \nabla v_\e^i \big|^2 \ \psi_i \dx \nonumber \\
  &\hspace{-0.45cm}= \limsup \limits_{\e \to 0}\Sum_{i=1}^k\int_{\omega_i} \Big(A_\e^{-1} \xi_\e \cdot \xi_\e - 2\xi_\e \cdot \nabla v_\e^i + A_\e \nabla v_\e^i \cdot \nabla v_\e^i \Big) \  \psi_i\dx \nonumber\\
&\hspace{-0.45cm}= \Sum_{i=1}^k\int_{\omega_i} \Big(A_*^{-1} \xi \cdot \xi - 2\xi \cdot \nabla v^i + A_* \nabla v^i \cdot \nabla v^i \Big) \psi_i\dx \nonumber \\
&\hspace{-0.45cm}= \Sum_{i=1}^k\int_{\omega_i}  \big[A_* \big(A_*^{-1} \xi - \lambda_i \big) \cdot \big(A_*^{-1} \xi - \lambda_i \big)\big] \psi_i\dx \leq \beta \Sum_{i = 1}^k \int_{\omega_i} |A_*^{-1}\xi - \lambda_i|^2 \dx \leq \beta \delta .
\end{align}  This combined with \eqref{defPhi} and \eqref{I} implies that \eq |I_\e| = O( \sqrt{\delta}).\label{I1}\qe

On the other hand, an integration by parts gives \begin{align} J_\e &= \Sum_{i = 1}^k \int_{\omega_i} \nabla v_\e^i \cdot \Div\big( \Phi_\e - \Phi_\e^\T \big) \psi_i \dx \nonumber \\
&=\Sum_{i = 1}^k \int_{\omega_i} \nabla v_\e^i \cdot \Div\big( \psi_i (\Phi_\e - \Phi_\e^\T) \big)  \dx - \Sum_{i = 1}^k \int_{\omega_i} \nabla v_\e^i \cdot \big( \Phi_\e^\T - \Phi_\e \big)\nabla \psi_i  \dx. \end{align} Since $\psi_i (\Phi_\e - \Phi_\e^\T)$ is an antisymmetric matrix, we have for any $i=1,\ldots,k$,\eq \int_{\omega_i} \nabla v_\e^i \cdot \Div\big( \psi_i (\Phi_\e - \Phi_\e^\T) \big) \dx = - \Sum_{k,l=1}^d \left \langle \frac{\partial^2 v_\e^i}{ \partial x_k \partial x_l } \, ,  \, \psi_i (\Phi_\e - \Phi_\e^\T)_{k,l} \right \rangle_{\mathscr C_c^\infty(\omega_i),\mathscr D(\omega_i)} =  0, \qe hence \eq J_\e = \Sum_{i = 1}^k \int_{\omega_i} \nabla v_\e^i \cdot \big( \Phi_\e - \Phi_\e^\T \big)\nabla \psi_i  \dx.\qe The Cauchy-Schwarz inequality gives \eq |J_\e|^2 \leq \Vert  \Phi_\e - \Phi_\e^\T \Vert^2_{L^2(\Omega)^{d\times d}} \ \sup \limits_{1 \leq i \leq k} \Vert \nabla \psi_i\Vert_\infty^2 \ \Sum_{i = 1}^k \Vert \nabla v_\e^i\Vert_{L^2(\omega_i)}^2 .\qe Then, since $k$ and $(\psi_i)_{1\leq i \leq k}$ are independent of $\e$, there exists $C_\delta >0$ such that \eq |J_\e| \leq C_\delta \Vert  \Phi_\e - \Phi_\e^\T \Vert_{L^2(\Omega)^{d\times d}}.\label{J1}\qe Combining \eqref{I1} and \eqref{J1} with \eqref{decomp}, we have, for any $\delta >0$, \eq \left|\big\langle \curl\big(A_\e^{-1} \xi_\e \big), \Phi_\e \big\rangle_{H^{-1}(\Omega)^{d \times d}, H_0^1(\Omega)^{d \times d}}\right| \leq C_\delta \ \Vert  \Phi_\e - \Phi_\e^\T \Vert_{L^2(\Omega)^{d\times d}} + O (\sqrt{\delta}). \qe Moreover, by the definition of $\Phi_\e$ in \eqref{defPhi} and Rellich's theorem, $\Phi_\e - \Phi_\e^\T$ converges strongly to $0$ in ${L^2(\Omega)^{d\times d}}$. Therefore, we get that for any $\delta >0$, \eq \limsup \limits_{\e \to 0} \left|\big\langle \curl\big(A_\e^{-1} \xi_\e \big), \Phi_\e \big\rangle_{H^{-1}(\Omega)^{d \times d}, H_0^1(\Omega)^{d \times d}}\right| \leq C \sqrt{\delta}, \qe which implies \eqref{rotcompact} due to the arbitrariness of $\delta > 0$ and \eqref{defPhi}.

\bigskip

Finally, let us prove that \eqref{rotcompact} implies the equality in \eqref{lem-ineq}. As $|A_\e^{-1}\xi_\e| \leq \beta |\xi_\e|$ is a bounded sequence in $L^2(\Omega)$, the following convergence holds up to a subsequence \eq \eta_\e:=A_\e^{-1} \xi_\e \harpoon \eta \quad \text{weakly in } L^2(\Omega)^d. \qe  By the div-curl lemma and \eqref{rotcompact}, we have \eq \eta_\e \cdot \xi_\e = A_\e^{-1} \xi_\e \cdot \xi_\e \harpoon \eta \cdot \xi \quad \text{in } \mathscr D'(\Omega) \label{conv01},\qe\eq \eta_\e \cdot A_\e \nabla v_\e^\lambda \harpoon \eta \cdot A_* \lambda \quad \text{in } \mathscr D'(\Omega)\label{conv03}.\qe Moreover, since $A_\e$ is symmetric, we have \eq \eta_\e \cdot A_\e \nabla v_\e^\lambda = \xi_\e \cdot \nabla v_\e^\lambda \harpoon \xi \cdot \lambda \quad \text{in } \mathscr D'(\Omega).\label{conv02}\qe From \eqref{conv03} and \eqref{conv02}, we deduce that for any $\lambda \in \R^d$, \eq \eta \cdot A_* \lambda = A_* \eta \cdot \xi = \xi \cdot \lambda \quad \text{a.e. in } \Omega,\qe which implies that \eq \eta = A_*^{-1}\xi \quad \text{a.e. in } \Omega.\label{lastone}\qe We conclude to the equality in \eqref{lem-ineq} combining \eqref{lastone} with \eqref{conv01}. \qed

\subsection{Higher-order terms}

In this section, we try to extend the inequality \eqref{ineq} when the magneto-resistance is replaced by any term of even-order in the expansion of the perturbed resistivity. We first establish an inequality (opposite to \eqref{ineq}) satisfied by the fourth-order term of the resistivity assuming that the Hall matrix is zero. Then, we prove that the positivity is not conserved for even-orders greater than two.

For the sake of simplicity, we lighten the notation of Remark \ref{Remconv}: for any functional space $H$, any integer $k$, any $k$-linear form $g^k:  \left(\R^3\right)^k \rightarrow H$ and any $h \in \R^3$, the $k^{\text{th}}$-order $g^k(h, \ldots,h)$ is symply denoted $g^k(h)$.
\label{exp-p1ext}
Let $n \geq 4$ be an integer. Assume that the conductivity satisfies the regularity condition \eqref{cond-reg} for any multi-index $|\nu | \leq  n$. As a consequence, the conductivity $\sigma_\e(h)$, the resistivity $\rho_\e(h) = \sigma_\e(h)^{-1}$ and the associated homogenized quantities $\sigma_*(h)$, $\rho_*(h)$ satisfy the $n^{\text{th}}$-order expansions in $h$\eq \hspace{-0.25cm} \left \{\begin{array}{l l} \vspace{0.3cm} \sigma_\e(h) =  \sigma_\e + \cdots + \sigma_\e^n(h) + o_{L^\infty(\Omega)^{3 \times 3}}(|h|^n), & \rho_\e(h) = \rho_\e + \cdots + \rho_\e^n(h ) + o_{L^\infty(\Omega)^{3 \times 3}}(|h|^n),  \\
\sigma_*(h) = \sigma_* + \cdots + \sigma_*^n(h) + o_{L^\infty(\Omega)^{3 \times 3}}(|h|^n), & \rho_*(h) = \rho_* + \cdots + \rho_*^n(h ) + o_{L^\infty(\Omega)^{3 \times 3}}(|h|^n),\end{array} \label{expext}\right.\qe where for any $h \in \R^3$, the matrices $\sigma_{\e / *}^k (h)$, $\rho_{\e / *}^k(h)$ are symmetric for even $k$ and antisymmetric for odd $k$. Note that with the notations of Section \ref{sect1} we have \eq \sigma_{\e / *}^1(h) = \mathscr E(S_{\e / *} h), \quad \rho_{\e / *}^1(h) = \mathscr E(R_{\e / *} h), \quad \sigma_{\e / *}^2(h) = \mathscr N_{\e / *}(h,h), \quad \rho_{\e / *}^2(h) = \mathscr M_{\e / *}(h,h).\label{hallext1}\qe Similarly to \eqref{exp-U} and \eqref{exp-P}, by the above regularity condition, the coercivity of $\sigma_\e(h)$ and the Meyers estimate \cite{Mey}, the potential $U_\e(h)$ and the corrector $P_\e(h)$ admit the following $n^{\text{th}}$-order expansions in $h$,\begin{align} & U_\e(h) = U_\e^0 + U_\e^1(h) + \cdots + U_\e^n(h) + o_{W^{1,p}(\Omega)^3}(|h|^n),\label{exp-Uext} \\ & P_\e(h) = P_\e^0 + P_\e^1(h) +\cdots + P_\e^n(h) + o_{L^p(\Omega)^{3 \times 3}}(|h|^n).\label{exp-Pext}\end{align}

\subsubsection{Fourth-order term with zero Hall matrix}

 \noindent We can now state the following result for the fourth-order term:

\begin{Prop} Assume that \eqref{cond-sym} and \eqref{cond-reg} for $|\nu| \leq 4$ are satisfied and that the norms of $\sigma_\e^2$, $\sigma_\e^3$ and $\sigma_\e^4$ are bounded in $L^\infty(\Omega)$. Then, in the absence of Hall effect (\textit{i.e.} $\rho_\e^1 = 0$), we have, for any $h \in O$, \eq
   \sigma_* \, \mathscr \rho_*^4(h) \, \sigma_* \leq \lim \limits_{w-L^1(\Omega)} \big( \sigma_\e P_\e^0\big)^\text{T} \mathscr \rho_\e^4(h) \big( \sigma_\e P_\e^0\big). \label{ineq4}\qe Moreover, \eqref{ineq4} is an equality if and only if \eq \mathrm{Curl} \left( \rho_\e^2(h) \sigma_\e P_\e^0 \right) \text{ lies in a compact subset of } H^{-1}(\Omega)^{3 \times 3 \times 3}. \label{condRot4}\qe
\label{4}
\end{Prop}

\begin{proof} The proof follows the framework of Section \ref{sect1} and \ref{sect2}. We first establish like in Theorem \ref{limites} a new expression of the difference of the two terms of \eqref{ineq4} through relations similar to Proposition~\ref{relations}. We then apply Lemma \ref{divcurl} to this new expression.

Let $h \in O$. As $\sigma_\e^1 = 0$, by Proposition \ref{relations} and \eqref{hallext1}, we have \eq \sigma_*^1 = \rho_\e^1 = \rho_*^1 = 0.\label{nullHall}\qe Considering the expansion at the fourth-order of $\sigma_\e(h) \rho_\e(h) = I_3$, we obtain similarly to the proof of Proposition \ref{relations}, for any $h \in O$, \eq \sigma_\e \, \rho_\e^4(h) + \sigma_\e^1(h) \, \rho_\e^3(h) + \sigma_\e^2(h) \, \rho_\e^2(h) + \sigma_\e^3(h) \, \rho_\e^1(h) + \sigma_\e^4(h) \, \rho_\e = 0, \qe which gives, by \eqref{nullHall}, \eq \sigma_\e \, \rho_\e^4(h) \, \sigma_\e = - \, \sigma_\e^2(h) \, \rho_\e^2(h) \, \sigma_\e - \sigma_\e^4(h).\label{eq4'}\qe Using again \eqref{nullHall} with \eqref{hallext1}, \eqref{relaN} can be rewritten, for any $h \in O$, \eq \sigma_\e^2(h) = \sigma_\e \, \rho_\e^2(h) \, \sigma_\e. \label{relaN2}\qe Combining \eqref{eq4'} with \eqref{relaN2}, we obtain \eq \sigma_\e \, \rho_\e^4(h) \, \sigma_\e = \sigma_\e^2(h) \, \sigma_\e^{-1} \, \sigma_\e^2(h) - \sigma_\e^4(h).\label{eq4''}\qe Similarly, we have for the homogenized quantities \eq \sigma_* \, \rho_*^4(h) \, \sigma_* = \sigma_*^2(h) \, \sigma_*^{-1} \, \sigma_*^2(h) - \sigma_*^4(h).\label{eq4''*}\qe

Taking into account the expansions \eqref{expext} and \eqref{exp-Pext}, we have (writing only the second and the fourth-order terms) like in \eqref{exp-sP} \eq \begin{array}{l}
   \vspace{0.2cm}\sigma_\e(h) P_\e(h)  =  \cdots + \big( \sigma_\e P_\e^2(h) + \sigma_\e^1 P_\e^1(h) + \sigma_\e^2(h) P_\e^0 \big) + \cdots  \\
    + \, \big( \sigma_\e P_\e^4(h) + \sigma_\e^1(h) P_\e^3(h) +\sigma_\e^2(h) P_\e^2(h) +\sigma_\e^3(h) P_\e^1(h) + \sigma_\e^4(h) P_\e^0\big) + o_{L^p(\Omega)^{3 \times 3}}(|h|^4), \label{exp-sPext'} 
\end{array} \qe that is by \eqref{nullHall} \begin{eqnarray}
   \sigma_\e(h) P_\e(h) & = & \cdots + \big( \sigma_\e P_\e^2(h) + \sigma_\e^2(h) P_\e^0 \big) + \cdots  \nonumber \\
   & & \hspace{-0.65cm}+ \, \big( \sigma_\e P_\e^4(h) +\sigma_\e^2(h) P_\e^2(h) +\sigma_\e^3(h) P_\e^1(h) + \sigma_\e^4(h) P_\e^0\big) + o_{L^p(\Omega)^{3 \times 3}}(|h|^4). \label{exp-sPext} 
\end{eqnarray}  By virtue of Remark \ref{Remconv}, using the properties \eqref{def-corr}-\eqref{conv-P} satisfied by the corrector $P_\e(h)$ in the expansions \eqref{exp-Pext}, \eqref{exp-sPext} and \eqref{expext}, we get that\begin{equation} \; \hspace{-0.3cm}\left\{\!\! \begin{array}{r c l l}
\vspace{0.2cm}P_\e^4(h) & \harpoon & 0  & \text{weakly in } L^{p}(\Omega)^{3 \times 3}, \\
\vspace{0.2cm}\sigma_\e P_\e^2(h) + \sigma_\e^2(h) P_\e^0 & \harpoon & \sigma_*^2(h)  & \text{weakly in } L^{p}(\Omega)^{3 \times 3}, \\
\sigma_\e P_\e^4(h) +\sigma_\e^2(h) P_\e^2(h) +\sigma_\e^3(h) P_\e^1(h) + \sigma_\e^4(h) P_\e^0 & \harpoon & \sigma_*^4(h)  & \text{weakly in } L^{p}(\Omega)^{3 \times 3}.
\end{array} \right.\;\;\;\label{conv3ext}\end{equation} and \begin{equation} \; \hspace{-0.26cm}\left\{\!\! \begin{array}{r c l}
 \vspace{0.2cm}\Div\big(\sigma_\e P_\e^2(h) + \sigma_\e^2(h) P_\e^0\big) & = & \Div\big(\sigma_*^2(h)\big),\\
\Div\big(\sigma_\e P_\e^4(h) +\sigma_\e^2(h) P_\e^2(h) +\sigma_\e^3(h) P_\e^1(h) + \sigma_\e^4(h) P_\e^0 \big) & = & \Div\big(\sigma_*^4(h)\big) 
\end{array} \quad \text{ in } \mathscr D'(\Omega)^{3 \times 3}. \right. \label{div2ext}\;\;\;\end{equation}Moreover, from $ \sigma_\e^1 = 0$, \eqref{nullHall} and \eqref{div2} $\sigma_\e  P_\e^1(h)$ is a divergence free function. Then the div-curl lemma implies that $ \big(\sigma_\e P_\e^1(h) \big)^\T P_\e^1(h)$ converges to $0$ in $L^{p/2}(\Omega)^{3 \times 3}$. This combined with $\sigma_\e \in \mathcal M(\alpha,\beta;\Omega)$, we get that \eq P_\e^1(h) \longrightarrow 0  \quad \text{strongly in } L^{p}(\Omega)^{3 \times 3}.\label{strongP1}\qe Taking into account \eqref{conv3} and \eqref{div2ext}, the div-curl lemma implies the convergence \eq \big(P_\e^0\big)^\T\big(\sigma_\e P_\e^4(h) +\sigma_\e^2(h) P_\e^2(h) +\sigma_\e^3(h) P_\e^1(h) + \sigma_\e^4(h) P_\e^0 \big) \harpoon \sigma_*^4(h)  \quad \text{weakly in } L^{p/2}(\Omega)^{3 \times 3}.\label{conv4}\qe Moreover by \eqref{conv3}, \eqref{conv3ext}, \eqref{div2} and the symmetry of $\sigma_\e$ the div-curl lemma yields \eq \big(P_\e^0\big)^\T\big(\sigma_\e P_\e^4(h)\big) = \big(\sigma_\e P_\e^0\big)^\T P_\e^4(h) \harpoon 0  \quad \text{weakly in } L^{p}(\Omega)^{3 \times 3}.\label{toto}\qe Hence, combining \eqref{strongP1} and \eqref{toto} in \eqref{conv4} we obtain \eq \sigma_*^4(h) = \lim \limits_{w-L^1(\Omega)} \Big[\big(P_\e^0\big)^\T\sigma_\e^2(h) P_\e^2(h) + \big(P_\e^0\big)^\T \sigma_\e^4(h) P_\e^0\Big].\label{conv4w}\qe Taking into account \eqref{Prop-P}, and \eqref{div2ext} the div-curl lemma implies that \eq \big(P_\e^2(h)\big)^\T \big(\sigma_\e P_\e^2(h) + \sigma_\e^2(h) P_\e^0 \big) \harpoon 0 \quad \text{ weakly in }L^{p/2}(\Omega)^{3 \times 3}, \qe hence \begin{align}
   & \lim \limits_{w-L^1(\Omega)} \big(\sigma_\e P_\e^2(h) + \sigma_\e^2(h) P_\e^0 \big)^\T  \sigma_\e^{-1} \big(\sigma_\e P_\e^2(h) + \sigma_\e^2(h) P_\e^0 \big) \nonumber \\ & =  \lim \limits_{w-L^1(\Omega)} \Big[ \big(P_\e^0 \big)^\T \sigma_\e^2(h) P_\e^2(h) + \big(P_\e^0 \big)^\T \sigma_\e^2(h) \sigma_\e^{-1} \sigma_\e^2(h)P_\e^0 \Big].  \label{convinterext}
\end{align} Combining the equalities \eqref{eq4''}, \eqref{eq4''*} with \eqref{conv4w}, \eqref{convinterext} and the symmetry of $\sigma_\e$ and $\sigma_\e^2(h)$, we obtain that \begin{align}
&\hspace{-0.6cm} \sigma_* \, \mathscr \rho_*^4(h) \, \sigma_*  - \lim \limits_{w-L^1(\Omega)} \big( \sigma_\e P_\e^0\big)^\text{T} \mathscr \rho_\e^4(h) \big( \sigma_\e P_\e^0\big) \nonumber\\
& \hspace{-0.8cm}= \sigma_*^2(h) \, (\sigma_*)^{-1} \, \sigma_*^2(h) - \sigma_*^4(h) -\lim \limits_{w-L^1(\Omega)} \Big[ \big(P_\e^0\big)^\T \sigma_\e^2(h) \sigma_\e^{-1} \sigma_\e^2(h)P_\e^0 - \big(P_\e^0\big)^\T \mathscr \sigma_\e^4(h)P_\e^0 \Big]\nonumber  \nonumber\\
& \hspace{-0.8cm}= \sigma_*^2(h) \, (\sigma_*)^{-1} \, \sigma_*^2(h) - \lim \limits_{w-L^1(\Omega)} \Big[ \big(P_\e^0 \big)^\T \sigma_\e^2(h) \sigma_\e^{-1} \sigma_\e^2(h)P_\e^0 + \big(P_\e^0 \big)^\T \sigma_\e^2(h) P_\e^2(h)\Big] \nonumber \nonumber\\
& \hspace{-0.8cm}= \sigma_*^2(h) \, (\sigma_*)^{-1} \, \sigma_*^2(h) - \lim \limits_{w-L^1(\Omega)} \big(\sigma_\e P_\e^2(h) + \sigma_\e^2(h) P_\e^0 \big)^\T  \sigma_\e^{-1} \big(\sigma_\e P_\e^2(h) + \sigma_\e^2(h) P_\e^0 \big).\label{diffext}
\end{align}
 
 \noindent Let $\lambda \in \R^3$. We apply Lemma \ref{divcurl} with $A_\e:=\sigma_\e$,  $\xi_\e:=\big(\sigma_\e P_\e^2(h) + \sigma_\e^2(h) P_\e^0 \big) \lambda$ which has a compact divergence by \eqref{div2ext} and converges weakly in $L^p(\Omega)^{3 \times 3}$ to $\xi:= \sigma_*^2(h)\lambda$ by \eqref{conv3ext}. Thus, with the notation of Lemma \ref{divcurl}, we have  \eq \zeta = \lim \limits_{w-L^1(\Omega)} \big( \sigma_\e^2(h) P_\e^0 + \sigma_\e P_\e^2(h)\big)^\text{T} \sigma_\e^{-1} \big( \sigma_\e^2(h) P_\e^0 + \sigma_\e P_\e^2(h)\big)\lambda \cdot \lambda. \qe As $\sigma_*^2(h)$ is symmetric, it follows that for any $\lambda \in \R^3$, \eq \lim \limits_{w-L^1(\Omega)} \big( \mathscr \sigma_\e^2(h) P_\e^0 + \sigma_\e P_\e^2(h)\big)^\text{T} \sigma_\e^{-1} \big( \sigma_\e^2(h) P_\e^0 + \sigma_\e P_\e^2(h)\big) \lambda \cdot \lambda \geq \sigma_*^2(h) \sigma_*^{-1} \sigma_*^2(h) \lambda \cdot \lambda. \label{tty1'ext}\qe Using the fact that $P_\e^2(h)$ is a gradient and the equality \eqref{relaN2}, \eqref{tty1'ext} is an equality if and only if \eq \mathrm{curl} \ \! \big( A_\e^{-1} \xi_\e \big) = \mathrm{curl} \ \! \big( \sigma_\e^{-1} \sigma_\e^2(h)  P_\e^0 \lambda \big) = - \, \mathrm{curl} \ \! \big( \rho_\e^2(h) \sigma_\e P_\e^0 \lambda\big) \label{tty2'ext}\qe lies in a compact subset of $H^{-1}(\Omega)^{3\times 3}$. This concludes the proof due to the arbitrariness of $\lambda$.  
\end{proof}

\subsubsection{An example with changes of sign}
 \label{counter}
 
 \noindent In this section we build a rank-one laminate which shows that the inequality \eqref{ineq} (or its inverse) cannot be extended to higher even-order terms.

Let $p \in \mathbb N^*$. Define the perturbed conductivity \eq \sigma(h):=\chi \, \sigma_1(h) + (1-\chi) \, \sigma_2(h),\label{lami}\qe where $\chi$ is a characteristic function. For $i=1,2$, the conductivities $\sigma_i(h)$ belong to $\mathscr M(\alpha,\beta;\Omega)$ and the resistivities $\rho_i(h) = \sigma_i(h)^{-1}$ satisfy the $n^{\text{th}}$-order expansions in $h$\eq  \vspace{0.3cm} \sigma_i(h) =  \sigma_i + \cdots + \sigma_i^{2p}(h) + o(|h|^{2p}), \quad \rho_i(h) = \rho_i + \cdots + \rho_i^{2p}(h ) + o(|h|^{2p}), \qe where for any $h \in \R^3$, the matrices $\sigma_i^k (h)$, $\rho_i^k(h)$ are symmetric for even $k$ and antisymmetric for odd $k$.

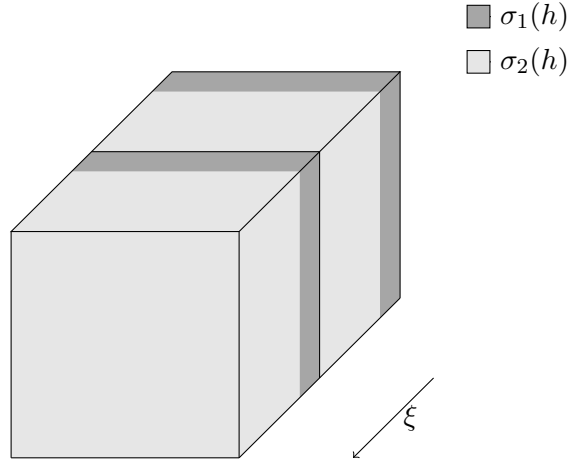
\begin{figure}[H]
\centering
\leavevmode
\begin{tikzpicture}[math3d, scale=3]
% face inferieure ABCD, D=O
\fill[color= gray!70] (2,2,2) -- (2,2.1,2) -- (2,2.1,1.9) -- (2,2,1.9)-- cycle;
\draw (2,2,2) -- (2,2.1,2) -- (2,2.1,1.9) -- (2,2,1.9)-- cycle;
\draw[-] (2,2.1,1.95) -- (2,2.1,1.95)
node[midway, right]{$\sigma_1(h)$};
\fill[color= gray!20] (2,2,1.8) -- (2,2.1,1.8) -- (2,2.1,1.7) -- (2,2,1.7)-- cycle;
\draw (2,2,1.8) -- (2,2.1,1.8) -- (2,2.1,1.7) -- (2,2,1.7)-- cycle;
\draw[-] (2,2.1,1.75) -- (2,2.1,1.75)
node[midway, right]{$\sigma_2(h)$};
\coordinate (a2) at (2,0,0); 
\coordinate (b2) at (2,1,0);
\coordinate (c2) at (0,1,0);
\coordinate (d2) at (0,0,0);
\coordinate (a) at (1,0,0);
\coordinate (b) at (1,1,0);
\coordinate (c) at (0,1,0);
\coordinate (d) at (0,0,0);
% face superieure EFGH (obtenue par translation de ABCD)
\coordinate (e) at (1,0,1);
\coordinate (f) at (1,1,1);
\coordinate (g) at (0,1,1);
\coordinate (h) at (0,0,1);
% face inferieure ABCD, D=O
\coordinate (a1) at (1,0,1);
\coordinate (b1) at (0.66,0,1);
\coordinate (c1) at (0.66,0.33,1);
\coordinate (d1) at (1,0.33,1);
% face superieure EFGH (obtenue par translation de ABCD)
\coordinate (e1) at (1,0,0.66);
\coordinate (f1) at (0.66,0,0.66);
\coordinate (g1) at (0.66,0.33,0.66);
\coordinate (h1) at (1,0.33,0.66);
% les 3 aretes cachees, au fond, partant de O
% le centre de gravite
%\draw ($ 0.33*(a) + 0.33*(f) + 0.33*(h) $) node{$\bullet$};
%\draw[dashed] (e) -- (c);
%\draw (h) -- (f) -- (a) -- cycle;
%\draw (h) -- (1,1/2,1/2) ;
%\draw (a) -- (1/2,1/2,1) ;
%\draw (f) -- (1/2,0,1/2) ;
%\fill[fill=gray!20, opacity=0.7] (a) -- (b) -- (f) --(e) -- cycle; % avant
%\foreach \row in {0, 0.01,...,1} { \foreach \col in {0.166,0.166 + 0.33,0.166 + 0.66} { \foreach %\high in {0.166,0.166 + 0.33,0.166 + 0.66} {
%\draw[dashed] (\row ,0, \high) -- (\row ,1, \high) -- cycle;
%\draw[dashed] (0 ,\col , \high) -- (1 ,\col , \high) -- cycle;
%\draw[dashed] (\row , \col , 0) -- (\row ,\col , 1) -- cycle;
%\draw [domain=3 *pi/2:5*pi/2, samples=20, smooth]
%plot[dashed,color=gray20!] (\row , \col + 0.05 * {cos(\x r)}, \high + 0.05 * {sin(\x r)}) ;
%}}}
\fill[color= gray!70] (0,0,1) -- (0,1,1)--(0,1,0) --(0.25,1,0)--(0.25,1,1)--(0.25,0,1)-- cycle;
\fill[color= gray!20] (0.25,0,1) -- (0.25,1,1)--(0.25,1,0) --(1,1,0)--(1,1,1)--(1,0,1)-- cycle;
\fill[color= gray!70] (1,0,1) -- (1,1,1)--(1,1,0) --(1.25,1,0)--(1.25,1,1)--(1.25,0,1)-- cycle;
\fill[color= gray!20] (1.25,0,1) -- (1.25,1,1)--(1.25,1,0) --(2,1,0)--(2,1,1)--(2,0,1)-- cycle;
\fill[color= gray!20] (2,0,0) -- (2,1,0)--(2,1,1) --(2,0,1)-- cycle;
\draw (2,0,0)--(2,1,0)--cycle;
\draw (0,1,0)--(2,1,0)--cycle;
\draw (2,0,1)--(2,1,1)--(0,1,1)--(0,0,1)--cycle;
\draw (2,0,0)--(2,0,1)--cycle;
\draw (2,1,0)--(2,1,1)--cycle;
\draw (0,1,0)--(0,1,1)--cycle;
\draw (1,0,1)--(1,1,1)--cycle;
\draw (1,1,0)--(1,1,1)--cycle;
\draw[->] (1,1.5,0) -- (2,1.5,0)
node[midway, right]{$\xi$};

\end{tikzpicture}
\caption{A three-dimensional rank-one laminate}
\label{figlami}
\end{figure}

We have the following result:

\begin{Prop} 
For any even integer $p \geq 2$, there exists a rank-one laminate such that for any $h$ the matrix \eq
   \mathscr D^{(2 p)}(h) := \sigma_* \mathscr \rho_*^{(2p)}(h) \sigma_* - \Big\langle \big( \sigma P^0\big)^\text{T} \mathscr \rho^{(2p)}(h) \big( \sigma P^0\big)\Big\rangle\label{Dcx}\qe is neither non-positive, nor non-negative. \label{cx}
\end{Prop}

\begin{proof} We consider the particular case of \eqref{lami} where the magnetic field is $h = h_3 \, e_3$, $\chi$ is a $1$-periodic function only depending on $x_1$,  and \eq \sigma_1(h) = \theta^{-1} \, I_3 + \theta^{-1} \, \mathscr E(h), \quad \sigma_2(h) = \alpha_2 \, I_3 + \mathscr E(h), \quad \text{with} \quad \theta:=\langle \chi \rangle \in (0,1), \quad  \alpha_2 >0. \label{partcase}\qe The laminate corrector $P(h)$ is explicitly given by (see, \textit{e.g.}, \cite{Bricor}) \eq P(h) = \chi \, P_1(h) + (1-\chi) \, P_2(h),\qe where for $i=1$, $2$, \eq P_i(h) = I_3 + \displaystyle \frac{(1-\theta)^{(2-i)} \, (-\theta)^{i-1}}{1-\theta + \theta^2 \alpha_2} \begin{pmatrix}
\theta \, \alpha_2 -1 & (1-\theta) \, h_3 & 0 \\
0 & 0& 0 \\
0& 0 & 0
\end{pmatrix}. \qe The homogenized conductivity is defined by \eq \sigma_*(h) = \big \langle \chi \, \sigma_1(h) P_1(h) + (1- \chi) \,  \sigma_2(h) P_2(h) \big \rangle, \qe which yields\eq \sigma_*(h) = \begin{pmatrix} \vspace{0.5cm} \displaystyle \frac{\alpha_2}{1-\theta + \theta^2 \, \alpha_2} & - \displaystyle \frac{1-\theta + \theta \, \alpha_2}{1-\theta + \theta^2 \, \alpha_2} \, h_3 & 0 \\ \vspace{0.5cm} \displaystyle \frac{1-\theta + \theta \, \alpha_2}{1-\theta + \theta^2 \, \alpha_2} \, h_3 & 1+ (1-\theta) \, \alpha_2 + \displaystyle \frac{(1-\theta)^3}{1-\theta + \theta^2 \, \alpha_2} \, h_3^2 & 0 \\ 0 & 0 & 1+ (1-\theta) \, \alpha_2 \end{pmatrix}.\qe Inverting this matrix, we obtain the homogenized resistivity \eq \rho_*(h) =  \begin{pmatrix} \vspace{0.3cm}\displaystyle \frac{b \, (1-\theta + \theta^2 \, \alpha_2)^2  + (1-\theta)^3 \, h_3^2}{b \, \alpha_2 + c \, h_3^2 } & \displaystyle \frac{1- \theta + \theta \, \alpha_2}{b \, \alpha_2 + c \, h_3^2} & 0 \\
\vspace{0.3cm} -\displaystyle \frac{1- \theta + \theta \, \alpha_2}{b \, \alpha_2 + c \, h_3^2} & \displaystyle \frac{\alpha_2}{b \, \alpha_2  + c \, h_3^2} & 0 \\ 0 & 0 & b^{-1}\end{pmatrix},\text{ where} \left\{  \hspace{-0.2cm} \begin{array}{l} b := 1+ (1-\theta), \\
c := 1-\theta + \alpha_2. \end{array} \right. \label{homrho} \qe
Expanding the quantities $\rho_i(h) = \sigma_i(h)^{-1}$,  we obtain the expressions \eq \rho_1^{(2p)}(h) = (-1)^p \, \theta \,h_3^{2p} \, K \quad \text{and} \quad \rho_2^{(2p)}(h) = (-1)^p \, \displaystyle \frac{h_3^{(2p)}}{\alpha_2^{2p+1}} \, K, \quad \text{where } K = \left(\begin{smallmatrix} 1 & 0 &0 \\
0 & 1 & 0 \\
0 & 0 & 0\end{smallmatrix} \right).\label{devrholoc}\qe Using \eqref{partcase}-\eqref{devrholoc}, we can compute $\mathscr D^{(2 p)} (h)$. All the coefficients in the matrix $\mathscr D^{(2 p)} (h)$ are zero except the entries \begin{equation} \left\{\!\! \begin{array}{l}
\vspace{0.3cm} \mathscr D^{(2 p)}(h)_{1,1} =  \displaystyle \frac{(-1)^p}{(1-\theta + \theta^2 \, \alpha_2)^2} \left [  \frac{(1-\theta + \theta \alpha_2)^2 \, (1-\theta + \alpha_2)^{p-1}}{\alpha_2^{p-1} \, \big( 1+ (1-\theta) \alpha_2 \big)^p} - \frac{1-\theta}{\alpha_2^{2 p -1}} - (\theta \, \alpha_2)^2\right] \, h_3^{2 p }, \\
 \mathscr D^{(2 p)}(h)_{2,2} := (-1)^p \left [ \displaystyle \frac{ (1-\theta + \alpha_2)^{p}}{\alpha_2^{p} \, \big( 1+ (1-\theta) \alpha_2 \big)^{p - 1}} - 1 - \frac{1-\theta}{\alpha_2^{2 p -1}}\right] \, h_3^{2 p }.
\end{array} \right.\label{DDDD}\end{equation}  

As $p \geq 2$, passing to the limit in \eqref{DDDD} successively when $\theta \to 0$ and $\alpha_2 \to \infty$, we obtain that \eq \lim \limits_{\theta \to 0} \mathscr D^{(2p)}(h) \underset{\alpha_2 \to \infty}{\sim} (-1)^p \, h_3^{2p}\begin{pmatrix} \displaystyle \frac{1}{\alpha_2^{p}} & 0 &0 \\
0 & \displaystyle  - 1 & 0 \\
0 & 0 & 0\end{pmatrix}.\qe Finally, when $\theta$ is small enough and $\alpha_2$ large enough, the matrix $\mathscr D^{(2 p)}(h)$ has a positive eigenvalue, and a negative eigenvalue, which concludes the proof.
\end{proof}  %\eq \rho_1(h) = \sigma_1(h)^{-1} = \begin{pmatrix} \vspace{0.3cm}\displaystyle \frac{\theta}{1+h_3^2} & \displaystyle \frac{\theta}{1+h_3^2} \, h_3& 0 \\

\begin{Rem}
The case $p=1$ confirms Theorem \ref{Thineq}. Indeed, we have \eqs \mathscr D^{(2)}(h) = \displaystyle \displaystyle \frac{(1+ (1-\theta) \alpha_2)^{-1}}{(1-\theta + \theta^2 \, \alpha_2)^2} \left [ \left(\frac{1-\theta}{\alpha_2} + \theta \, \big(\theta \alpha_2^2\big) \right) \, \big( (1-\theta) \alpha_2 + \theta \, \theta^{-1} \big) - (1-\theta + \theta \alpha_2)^2 \right] \, h_3^{2} \, (e_1 \otimes e_1), \qes which is positive by the Cauchy-Schwarz inequality. This formula is a particular case of formula \eqref{diffone} below.
\end{Rem}

\section{Case of equality for a few periodic structures} \label{exemples}

In this section we consider various periodic microstructures in the case of equality for \eqref{ineqper}. On the one hand, Section \ref{lay} provides an explicit expression of the difference between the two terms of \eqref{ineqper} for layered structures and thus the different cases of equality. On the other hand, Section \ref{cyli} only provides the cases of equality for columnar structures: condition \eqref{condRotper} would have consequences on the Hall matrix and the conductivity of the microstructure. We use the notations of Section \ref{secper}.

More precisely, we study the consequences of $D(h,h) = 0$ in \eqref{introD}. For a given averaged-value $\lambda \in \R^3$ of the electric field $e = P^0 \lambda$ in a composite conductor, we have the relations for the local current and the averaged-value of the current (see Remark \ref{electric}) \eq j = \sigma e = \sigma P^0 \lambda, \quad \langle j \rangle = \sigma_* \langle e \rangle = \sigma_* \lambda. \qe We set \eq \mathscr D(h,h):= \sigma_* \, \mathscr M_*(h,h) \, \sigma_* - \left\langle \big( \sigma P^0\big)^\text{T} \mathscr M(h,h) \big( \sigma P^0\big) \right \rangle, \label{diffDronde}\qe so that, by the symmetry of $\sigma_*$, it follows that \eq D(h,h) = \mathscr D(h,h) \lambda \cdot \lambda.\qe

\subsection{Periodic layered structures}
\label{lay}

In this section, we establish for a periodic layered structure depending on a direction $\xi \in \R^3$, $|\xi| = 1$, an exact formula for the difference between the effective magneto-resistance and the averaged local magneto-resistance.

Let $\sigma(h)$ be a perturbed conductivity in $\mathcal M(\alpha,\beta;\Omega)$ only depending on $\xi \cdot y$, and satisfying the expansion \eq \sigma(h)(y)= a(\xi \cdot y)\, I_3 + s(\xi \cdot y) \, \mathscr E(h) + \mathscr N(h,h)(\xi \cdot y) + o(|h|^2), \quad \text{for a.e. } y \in Y, \label{perper}\qe where $a : \R \to [\alpha,\beta]$ and $s : \R \to \R$ are $1$-periodic functions. By Proposition \ref{relations}, we have \eq R = r \, I_3 = - \displaystyle \frac{s}{a^2} \, I_3.\label{srone}\qe Considering the expansions \eqref{exp-cond*per}-\eqref{exp-Pper}, we can state a result precising Corollary \ref{Thineqper}:

\begin{Prop}
Consider a conductivity $\sigma(h)$ satisfying \eqref{perper} and the matrix-valued function $\mathscr D(h,h)$  defined by \eqref{diffDronde}. \begin{itemize} \item[$\bullet$] When $h$ is not parallel to $\xi$, we have \eq O(h)^\T \mathscr D(h,h) \, O(h) = \begin{pmatrix}
  \vspace{0.2cm} d_1 \, |h \times \xi|^2 & d_3 \, (h \cdot \xi) |h \times \xi| & 0 \\
  \vspace{0.2cm} d_3 \, (h \cdot \xi) |h \times \xi| & d_2 \, (h \cdot \xi)^2 & 0 \\
   0&0& d_2 \, (h \cdot \xi)^2
\end{pmatrix}, \label{diffone}\qe where $O(h)$ is the change-of-basis matrix from the canonical basis to\eq \mathscr{\hat{\mathscr B}} = \left(\xi, \frac{\xi \times (\xi \times h)}{|h \times \xi|}, \frac{h \times \xi}{|h \times \xi|}\right), \ \ \text{and} \ \ \left\{\!\! \begin{array}{l l}
\vspace{0.2cm} d_1:= \displaystyle \big \langle a^{-1} \big \rangle^{-2}\left[ \left \langle \displaystyle a \, r^2 \right \rangle -  \langle a\rangle^{-1} \left \langle \displaystyle a \, r\right \rangle^2 \right],\\
\vspace{0.2cm} d_2:= \displaystyle \left \langle \displaystyle a^3 \, r^2 \right \rangle -  \langle a\rangle^{-1} \langle a^2 \, r \rangle^2, \\
d_3:= \displaystyle \big \langle a^{-1} \big \rangle^{-1}\left[ \left \langle \displaystyle a^2 \, r^2 \right \rangle -  \langle a\rangle^{-1} \left \langle \displaystyle a \, r \right \rangle \langle a^2 \, r \rangle \right].
\end{array}  \;\;\;\right.\qe

\item[$\bullet$] When $h$ is parallel to $\xi$, we have \eq O^\T \mathscr D(h,h) \, O = d_2 \, (h \cdot \xi)^2 \begin{pmatrix}
0 & 0 &0 \\
0 & 1 & 0 \\
0 & 0 & 1
\end{pmatrix}, \label{difftwo}\qe where $O$ is the change-of-basis matrix from the canonical basis to an orthonormal basis \linebreak[4] $\mathscr{\hat{\mathscr B}} = (\xi ,u , v)$, for suitable $u,v \in \R^3$.
\end{itemize}

Moreover, $\mathscr D(h,h) = 0$ if and only if one of the following conditions holds: \begin{itemize}
\item[$\bullet$] $h = 0$;
\item[$\bullet$] $h \neq 0$ is orthogonal to $\xi$, and $r$ is a constant;
\item[$\bullet$] $h \neq 0$ is parallel to $\xi$, and $a r$ is a constant;
\item[$\bullet$] $h \neq 0$ is neither parallel to $\xi$, nor orthogonal to $\xi$, and $a r$, $r$ are constant. 
\end{itemize}
\label{onedirection}
\end{Prop}

\begin{proof} \textcolor{white}{t}

\smallskip

\noindent \textit{First case: $h \neq 0$ is not parallel to $\xi$.} By Corollary \ref{THconvper} we have \eq \mathscr D(h,h)= \left \langle \big( \mathscr E(S h) P^0 + \sigma P^1(h)\big)^\text{T} \sigma^{-1} \big( \mathscr E(S h) P^0 + \sigma P^1(h)\big) \right \rangle - \mathscr E(S_* h)^\text{T} \sigma_*^{-1} \mathscr E(S_* h). \label{Dhhnew}\qe 

For the sake of simplicity denote $O := O(h)$. Denoting by $\hat{\cdot}$ the quantities with respect to the new basis $\hat{\mathscr {B}}$, we have the following change-of-basis formulas respectively for the system of coordinates, the local conductivity, the zero and first-order terms in the expansion of the corrector and the local $S$-matrix defined by \eqref{exp-condper}: \eq \; \left\{\!\! \begin{array}{l l}
\vspace{0.2cm} \hat{y} = O^\T y, & \hat{h} = O^\T h, \\
\vspace{0.2cm} \hat{\sigma}\big( \hat{h} \big) = O^\T \sigma(h) \, O, & \hat{\sigma} = O^\T \sigma \, O, \\ \vspace{0.2cm} \hat{S} = O^\T S \, O & \text{(as a consequence of \eqref{lemalg})},\\
\hat{P}^0 = O^\T P^0 \, O, &  \hat{P}^1\big( \hat{h} \big) = O^\T P^1(h) \, O, \\
\end{array}  \;\;\;\right.\label{chcoord}\qe where the last equality is a consequence of the relation \eq \forall \, \lambda \in \R^3, \quad \hat{P}\big( \hat{h} \big) \hat{\lambda} = O^\T P(h) \lambda = O^\T P(h) \, O \, O^\T \lambda = O^\T P(h) \, O \, \hat{\lambda}, \quad \text{with } \hat{\lambda} = O^\T \lambda.\qe Due to Lemma 38 of \cite{Taropti} for the homogenized conductivity and to \eqref{lemalg} for the homogenized $S$-matrix defined by \eqref{exp-cond*per}, we have \eq \hat{\sigma}_*\big( \hat{h} \big) = O^\T \sigma_*(h) \, O, \quad \hat{\sigma}_* = O^\T \sigma_* \, O  \quad \text{and}\quad \hat{S}_* = O^\T S_* \, O.\label{chcoord*}\qe From these relations and \eqref{Dhhnew} it is easy to check that the difference term $\mathscr D(h,h)$ defined by \eqref{diffDronde} satisfies the relation \eq \mathscr{\hat{\mathscr D}}\big( \hat{h}, \hat{h}\big)  = O^\T \mathscr D(h,h) \, O,\qe where \eq \mathscr{\hat{\mathscr D}}\big( \hat{h}, \hat{h}\big) := \Big \langle \big( \mathscr E\big(\hat{S} \, \hat{h}\big) \hat{P}^0 + \hat{\sigma} \hat{P}^1\big(\hat{h})\big)^\text{T} \hat{\sigma}^{-1} \big( \mathscr E\big(\hat{S} \, \hat{h}\big) \hat{P}^0 + \hat{\sigma} \hat{P}^1\big(\hat{h}\big)\big) \Big \rangle - \mathscr E\big(\hat{S}_* \, \hat{h})^\text{T} \hat{\sigma}_*^{-1} \mathscr E\big(\hat{S}_* \, \hat{h} \big).\label{Dtilde}\qe Let us now compute $\mathscr{\hat{\mathscr D}}\big( \hat{h}, \hat{h}\big)$ in \eqref{Dtilde}. We have \eq h = (h \cdot \xi, -|h \times \xi|, 0)^\T = (
\hat{h}_1, \hat{h}_2,0)^\T. \label{defh}\qe By isotropy, we have $\hat{\sigma} = a \, I_3$ and $ \hat{S} = s \, I_3$. By the uniqueness of the solution of problem \eqref{def-corrper}-\eqref{condmoyper}, we have \eq \hat{P}^0 = \begin{pmatrix} \displaystyle\frac{\big \langle a^{-1} \big \rangle^{-1}}{{a}} & 0 & 0\\
0 & 1 & 0 \\ 0 &0 & 1\end{pmatrix} \quad \text{and} \quad \hat{P}^1\big(\hat{h}\big) = \begin{pmatrix} \displaystyle0 & 0 & - \displaystyle\frac{\displaystyle s-\big \langle a^{-1} \big \rangle^{-1} \left\langle \frac{s}{a}\right\rangle}{{a}}\, \hat{h}_2\\
0 & 0 & 0 \\ 0 &0 & 0\end{pmatrix}.\label{ttttt}\qe Hence, combining the equality \eqref{ttttt} with the classical periodic homogenization formula (see, \textit{e.g.}, \cite{tpeccot} for more details) \eq \hat{\sigma}_* = \big \langle \hat{\sigma} \, \hat{P}^0\big \rangle = \begin{pmatrix} \displaystyle \big \langle a^{-1} \big \rangle^{-1} & 0 & 0\\
0 & \langle a \rangle & 0 \\ 0 & 0 & \langle a \rangle \end{pmatrix}. \qe Moreover, we have \eq \label{calc1} \mathscr E\big(\hat{S} \, \hat{h} \big) \hat{P}^0 + \hat{\sigma} \, \hat{P}^1\big(\hat{h}\big) = s \, \mathscr E\big(\hat{h}\big) \, \hat{P}^0 + a \, \hat{P}^1\big(\hat{h}\big) = \begin{pmatrix} 0 & 0 & \vspace{0.3cm} \displaystyle \big \langle \sigma^{-1} \big \rangle^{-1} \left\langle \frac{s}{a}\right\rangle \hat{h}_2\\
\vspace{0.3cm} 0 & 0 & - s \, \hat{h}_1 \\ -\displaystyle \big \langle \sigma^{-1} \big \rangle^{-1}  \frac{s}{a} \, \hat{h}_2 & s \, \hat{h}_1 & 0 \end{pmatrix} .\qe Also by Corollary \ref{THconvper}, we obtain that \eq \hat{S}_* = \big \langle \text{Cof} \big( \hat{P}^0 \big) \hat{S} \big\rangle = \begin{pmatrix}
\langle s \rangle & 0 & 0 \\ 0 & \displaystyle \langle a^{-1}\rangle^{-1} \left\langle \displaystyle\frac{s}{a}\right\rangle & 0 \\ 
0 & 0 & \langle a^{-1}\rangle^{-1} \left\langle \displaystyle \frac{s}{a}\right\rangle
\end{pmatrix}, \qe and \eq \mathscr E\big (\hat{S}_* \, \hat{h} \big) = \begin{pmatrix} \vspace{0.3cm} 0 & 0 & \displaystyle \big \langle a^{-1} \big \rangle^{-1} \left\langle \frac{s}{a}\right\rangle \hat{h}_2\\
\vspace{0.3cm} 0 & 0 & - \langle s \rangle \, \hat{h}_1 \\  -\displaystyle \big \langle a^{-1} \big \rangle^{-1} \left\langle \frac{s}{a}\right\rangle \hat{h}_2 & \langle s \rangle \, \hat{h}_1 & 0 \end{pmatrix}. \label{calc2}\qe Putting \eqref{calc1}-\eqref{calc2} in \eqref{Dtilde}, we get that \eq \mathscr{\hat{\mathscr D}}\big( \hat{h}, \hat{h}\big) = \begin{pmatrix}
 \vspace{0.3cm} d_1 \, \hat{h}_2^2 & - d_3 \, \hat{h}_1 \, \hat{h}_2 & 0 \\
  \vspace{0.3cm} - d_3 \, \hat{h}_1 \, \hat{h}_2 & d_2 \, \hat{h}_1^2 & 0 \\
    0&0& d_2 \, \hat{h}_1^2
\end{pmatrix}, \qe where\eq \; \left\{\!\! \begin{array}{l l}
\vspace{0.2cm} d_1:= \displaystyle \big \langle a^{-1} \big \rangle^{-2}\left[ \left \langle \displaystyle \frac{s^2}{a^3} \right \rangle -  \langle a\rangle^{-1} \left \langle \displaystyle \frac{s}{a}\right \rangle^2 \right],\\
\vspace{0.2cm} d_2:= \displaystyle \left \langle \displaystyle \frac{s^2}{a} \right \rangle -  \langle a\rangle^{-1} \langle s \rangle^2, \\
d_3:= \displaystyle \big \langle a^{-1} \big \rangle^{-1}\left[ \left \langle \displaystyle \frac{s^2}{a^2} \right \rangle -  \langle a\rangle^{-1} \left \langle \displaystyle \frac{s}{a}\right \rangle \langle s \rangle \right].
\end{array}  \;\;\;\right.\qe We deduce \eqref{diffone} from \eqref{srone} and \eqref{defh}.
   
   \smallskip
   
\noindent \textit{Second case: $h \neq 0$ is parallel to $\xi$.} Then, we have $P_1\big(h\big) = 0$, and the computations are quite similar. 

\bigskip
   
\noindent \textit{Cases of equality.} When $h \neq 0$ is not parallel to $\xi$ but orthogonal to $\xi$, by \eqref{diffone} $\mathscr D(h,h) = 0$ implies that $d_1 = 0$. Thus, the equality \eq \left \langle \displaystyle a \, r^2 \right \rangle =  \langle a\rangle^{-1} \left \langle \displaystyle a \, r\right \rangle^2 \qe can be regarded as the case of equality in the Cauchy-Schwarz inequality satisfied by the functions $\sqrt{a}$ and $\sqrt{a} \, r$ in $L^2\big( [0,1] \big)$. Therefore, $\sqrt{a}$ is proportional to $\sqrt{a} \, r$, hence $r$ is constant. The converse is immediate. The other cases are similar.
\end{proof}

\subsection{Periodic columnar structures}
\label{cyli}

\subsubsection{The general case}
In this section, we consider columnar isotropic structures in the direction $y_3$. More precisely, the $Y$-periodic conductivity $\sigma(h)$ of \eqref{exp-condper} only depends on $y' = (y_1,y_2)$ with \begin{equation} \; \left\{\!\! \begin{array}{l l}
\vspace{0.1cm}\sigma(0) := \sigma(y')  \, I_3, \quad & \sigma \in L_\sharp^\infty\big((0,1)^2;[\alpha,\beta] \big),\\
\vspace{0.1cm}S := s(y') \, I_3, \quad &s \in L_\sharp^\infty\big((0,1)^2;\R \big), \\
\vspace{0.1cm}\mathscr N(h,h) := \mathscr N(h,h)(y'), \quad & \mathscr N(h,h) \in L_\sharp^\infty\big((0,1)^2;\R_s^{3 \times 3} \big). \\
\end{array} \right. \;\;\;\label{cylcond}\end{equation} Consequently, by Proposition \ref{relations} the expansion \eqref{exp-p1per} of $\rho(h)$ satisfies \begin{equation} \; \left\{\!\! \begin{array}{l l}
\vspace{0.2cm}\rho = \sigma(y')^{-1} \, I_3, \\
\vspace{0.2cm}R = r(y') \, I_3, & \text{with } r = - \sigma^{-2} \, s, \\
\vspace{0.1cm}\mathscr M(h,h) := \mathscr M(h,h)(y'), \quad & \mathscr M(h,h) \in L_\sharp^\infty\big((0,1)^2;\R_s^{3 \times 3} \big). \\
\end{array} \right. \;\;\;\end{equation}

We have the following result:

\begin{Prop} Consider a conductivity $\sigma(h)$ satisfying \eqref{cylcond} and set $h' = (h_1,h_2)$. Assume that \eq r^{-1} \in L^1(Y) \quad \text{and} \quad \big \langle (\sigma \, r )^{-1}\big \rangle \neq 0.\label{condr}\qe Then, $\mathscr D(h,h) =0$ (see \eqref{diffDronde}) is an equality if and only if one of the following conditions holds \begin{itemize}
      \item $h=0$;
      \item $h'=0$, $h_3 \neq 0$, and the Hall coefficient $r$ is constant;
      \item $h' \neq 0$, $h_3 = 0$ and there exist two positive functions $f$, $g$ in $L^\infty(\R)$, with $f^{-1}$, $g^{-1}$ in $L^\infty(\R)$, which are $h_i$-periodic for $i=1,2$, and a constant $C$ such that \eq \quad \sigma(y') =f(h' \cdot y') \ g (Jh' \cdot y') \quad \text{and} \quad  r(y') = \frac{C}{f(h' \cdot y')} \quad \text{a.e. } y' \in (0,1)^2,\label{cylsigmafg}\qe where $J := \left(\begin{smallmatrix} 0 & -1 \\ 1 & 0\end{smallmatrix}\right)$;
      \item $h' \neq 0$, $h_3 \neq 0$, the Hall coefficient $r$ is constant, and there exists a function $g$ in $L^\infty(\R)$ with $g^{-1}$ in $L^\infty(\R)$ which is $h_i$-periodic for $i=1,2$, such that \eq \sigma(y') = \ g (Jh' \cdot y') \quad \text{a.e. } y' \in (0,1)^2.\label{cylsigmag}\qe
      \end{itemize}
      
      \smallskip
      
Moreover, when $h_1 h_2 \neq 0$ and $h_1 / h_2 \notin \mathbb Q$, $\sigma$ and $r$ are constant.
\label{cyl}
\end{Prop}

\begin{Rem} The case $(h_1,h_2) = 0$ corresponds to the two-dimensional case in \cite{Brimag} (Theorem 2.4). In the case $(h_1,h_2) \neq (0,0)$, $h_3 = 0$, $f$ and $g$ are not unique. For example $f$ and $g$ can be chosen such that $\langle f^{-1}\rangle = 1$ to ensure the uniqueness. \label{uniqfg}
\end{Rem}

\noindent \textit{Proof of Proposition \ref{cyl}.} We work in the orthonormal basis $(f_1,f_2,e_3)$ defined by\eq (f_1,f_2):=\left \{ \begin{array}{c l} \displaystyle \vspace{0.2cm} \left( \frac{h'}{|h'|}, \frac{J h'}{|h'|} \right) & \text{if } h' \neq (0,0), \\  (e_1,e_2) & \text{if } h' = (0,0). \end{array} \right.\qe In the new basis, we have $ h = |h'| f_1 + h_3 e_3$. The associated system of coordinates is given by \eq \left \{ \begin{array}{l} \displaystyle \vspace{0.2cm} z_1 := \frac{h_1 y_1 + h_2 y_2}{|h'|}, \\  \displaystyle  \vspace{0.2cm} z_2 : = \frac{h_1 y_2 - h_2 y_1}{|h'|}, \\  z_3 := y_3, \end{array} \right. \ \text{if } h' \neq 0, \quad \ \text{and} \ \quad z = y, \ \text{if } h' = 0. \label{chgcoor}\qe We denote for $i=1,2$, $P^0 f_i:=\nabla u^i$, $P^0 e_3:=\nabla u^3$ and for $i=1,2,3$, $v^i(z) = u^i(y)$.

Since the gradient, the divergence and the curl are invariant by a change of orthonormal right-handed basis, by \eqref{condRotper} we have for any $i=1,2,3$, \eq 0 = \mathrm{curl} \left( (\sigma r) \mathscr{E}(h)  \nabla u^i \right) = \begin{pmatrix} \vspace{0.2cm} |h'| \, \partial_{z_2} \, \big( (\sigma r) \, \partial_{z_2} v^i \big) - \partial_{z_3} \big( (\sigma r) \, (h_3 \, \partial_{z_1} v^i - |h'|  \partial_{z_3} v^i) \big)\\ \vspace{0.2cm} - \, h_3 \, \partial_{z_3} \, \big( (\sigma r) \,  \partial_{z_2} v^i \big) - \, |h'| \, \partial_{z_1} \big( (\sigma r) \,  \partial_{z_2} v^i \big) \\ \partial_{z_1} \, \big( (\sigma r) \, (h_3 \partial_{z_1} v^i - |h'| \partial_{z_3} v^i) \big) + \, h_3 \, \partial_{z_2} \big( (\sigma r) \, \partial_{z_2} v^i \big)\end{pmatrix}.\label{DV'1}\qe As $v^i$, $\sigma$, $r$ are independent of $z_3$, \eqref{DV'1} reads as \eq \begin{pmatrix} \vspace{0.2cm} |h'| \, \partial_{z_2} \big( (\sigma r) \, \partial_{z_2} v^i \big)\\ \vspace{0.2cm} |h'| \,   \partial_{z_1} \, \big( (\sigma r) \, \partial_{z_2} v^i  \big) \\ h_3 \, \partial_{z_1} \big( (\sigma r) \, \partial_{z_1} v^i \big) + h_3 \, \partial_{z_2} \, \big( (\sigma r) \, \partial_{z_2} v^i \big)\end{pmatrix} = 0, \quad \text{for } i=1,2.\label{DV1}\qe

\smallskip

\noindent \textit{First case: $h' = 0$ and $h_3 \neq 0$.} We are led to the two-dimensional case of \cite{Brimagneto} with $h = h_3 \, e_3$. The key ingredient is the positivity of the determinant of the corrector $P_0$ due to Alessandrini and Nesi~\cite{AlNes}.

\smallskip

\noindent \textit{Second case: $h' \neq 0$ and $h_3 = 0$.} Without loss of generality, we can assume that $|h'| = 1$. The two first equalities of \eqref{DV1} give the existence of a constant $C$ such that \eq  (\sigma r)  \, \partial_{z_2} v^1 = C.\label{cstv}\qe Since $  \nabla u^1 = \partial_{z_1} v^1 f_1 + \partial_{z_2} v^1 f_2 + \partial_{z_3} v^1 e_3 $, we have  by \eqref{condmoyper}  \eq \langle \partial_{z_2} v^1 \rangle  = \langle \nabla u^1 \rangle \cdot f_2 = f_1 \cdot f_2 = 0\label{condcrois}.\qe By \eqref{condr}, since $0 < \alpha \leq \sigma \leq \beta$ and $r^{-1} \in L^1(Y)$, $(\sigma r)^{-1} \neq 0$ almost everywhere in $\R^2$. Combining \eqref{condcrois} with \eqref{cstv}, we get that \eq C \, \big \langle (\sigma r )^{-1} \big \rangle  =0, \label{derivdirec}\qe hence, $C=0$, which implies that $v^1$ is a function of $z_1$. On the other hand, the Alessandrini, Nesi~\cite{AlNes} result combined with $v^1 = v^1(z_1)$ yields \eq \det(P_0) = \partial_{z_1} v^1 \partial_{z_2} v^{2} - \partial_{z_1} v^{2} \underbrace{\partial_{z_2} v^{1}}_{= 0} =  \partial_{z_1} v^1 \partial_{z_2} v^{2} >0 \quad \text{a.e. in } Y. \label{positif}\qe Moreover, by \eqref{def-corrper} we have \eq 0 = \div \ \! \big( \sigma \nabla v^1\big) = \partial_{z_1} \big( \sigma \, \partial_{z_1} v^1 \big) \quad \text{in } \mathscr D'(\R^3),\qe which implies that $\sigma \, \partial_{z_1} v^1$ is a function of $z_2$. By \eqref{positif}, we may define the two measurable functions $f$, $g$ by \eq f(z_1):= \big(\partial_{z_1} v^1\big)^{-1} \quad \text{and} \quad g(z_2):=\sigma \, \partial_{z_1} v^1. \label{defphipsi}\qe Therefore, we get that $\sigma(y') = f(z_1) \, g(z_2)$. In particular, $f$, $g$ are $h_i$-periodic for $i=1,2$. Let us show that $f$, $g$, $f^{-1}$, $g^{-1}$ are bounded functions. Denote $\delta := \max( |h_1|,|h_2|) >0$. As $\alpha \leq \sigma \leq \beta$, we have by \eqref{defphipsi}, \eq \beta^2 \,f^{-2}(z_1) \geq \sigma^2 \, f^{-2}(z_1) = g^2(z_2) \geq \alpha^2 \, f^{-2}(z_1) >0 \quad \text{a.e. in } (z_1,z_2) \in (0,\delta)^2. \label{eqfog} \qe Integrating \eqref{eqfog} successively with respect to $z_1$ and $z_2$ on $(0,\delta)$, we get that \eq \; \left\{\!\! \begin{array}{l}
\vspace{0.2cm} \alpha^{-2} \, C_1 \geq f^{-2}(z_1) \geq \beta^{-2} \, C_1,\\
 \beta^2 \, \alpha^{-2} \, C_1 \geq g^2(z_2) \geq \alpha^{2} \, \beta^{-2} \, C_1,
\end{array}  \;\;\;\right.  \quad \text{with } C_1:= \fint_0^{\delta} g^2(z_2) \ \text{d}z_2 >0\qe that is $f,g, f^{-1}, g^{-1}$ are $L^\infty(\R)$ functions. Integrating the inequality  $\sigma(y') = f(z_1) \, g(z_2) \geq \alpha$ with respect to $z_2$ on $(0,\delta)$, we obtain that \eq \quad f(z_1) \, \int_0^\delta g(z_2)\ \text{d}z_2  \geq \delta \, \alpha  > 0 \quad \text{a.e. } z_1 \in \R, \qe that is $f$ has a constant sign. Moreover, like in \eqref{condcrois} we have \eq \langle f^{-1} \rangle = \langle \partial_{z_1} v^1 \rangle = \langle \nabla u^1 \rangle \cdot f_1= f_1 \cdot f_1 = 1\label{condcrois1}.\qe Hence $f$ is a positive function, so is $g$ by \eqref{defphipsi}. Then, by a uniqueness argument the expression of $\sigma$ implies that the potentials $v^i$, $i=1$, $2$, $3$, are given by \eq \; \left\{\!\! \begin{array}{l l}
\vspace{0.2cm} \partial_{z_1} v^1 = f^{-1}(z_1), & \partial_{z_2} v^1 = \partial_{z_3} v^1 = 0,\\
\vspace{0.2cm} \partial_{z_2} v^2 = \langle g^{-1} \rangle^{-1} \, g^{-1}(z_2), & \partial_{z_1} v^2 = \partial_{z_3} v^2 = 0, \\
v^3 = z_3.
\end{array}  \;\;\;\right.\label{defcorrcyl''}\qe The conditions \eqref{DV1} and \eqref{positif} give the existence of a constant $C$ such that \eq  r(y') =  \frac{C}{\sigma \ \partial_{z_2} v^2} = C \, \frac{ \langle g^{-1} \rangle}{f(z_1)}.\label{r}\qe Using the expressions \eqref{chgcoor} and $|h'| = 1$, we obtain \eqref{cylsigmafg}.

Conversely, if the conductivity and the Hall coefficient satisfy \eqref{cylsigmafg} with $\langle f^{-1} \rangle = 1$ (see Remark~\ref{uniqfg}), the potentials $v^i$, for $i=1$, $2$, $3$, are given by \eqref{defcorrcyl''}. Hence, it follows immediately \eqref{DV'1} and thus the case of equality. 

\smallskip

\noindent \textit{Third case: $h' \neq 0$ and $h_3 \neq 0$.} Considering the third equality of \eqref{DV1}, the first case shows that $r$ is constant in $Y$. Moreover, taking into account the first and second components of \eqref{DV1}, $\sigma, r$ takes the form \eqref{cylsigmafg} by the second case. Hence, $f$ is constant which gives \eqref{cylsigmag}. The converse is similar to the second case.

\smallskip

\noindent \textit{Case where $h_1 h_2 \neq 0$ and $h_1/h_2 \notin \mathbb Q$.} As $h_1,h_2 \neq 0$, we are in the second or third case of the proof. Let $i \in \{ 1,2\}$. We have proved that $u^i(y) = v^i(z)$ is a function of $z_i$. Moreover, $\varphi \, : \,  y \mapsto u^i(y)-f_i \cdot y = v^i(z) - z_i$ is a function in $H^1_\sharp(Y;\R)$. The function $\varphi$ has a continuous representative and is $h_j$-periodic for $j=1,2$. As $h_1/h_2 \notin \mathbb Q$, $\varphi$ is constant so is $\nabla v^i$. Therefore, by \eqref{defcorrcyl''}, $f,g$ are also constant. Finally by \eqref{cylsigmafg}, $\sigma$ and $r$ are constant.\qed 

\subsubsection{Four-phase checkerboard}

In the section, we consider a four-phase checkerboard columnar structure. Let $\alpha_1$, $\alpha_2$, $\alpha_3$, $\alpha_4$ be positive numbers. Consider the $Y$-periodic conductivity only depending on $y' = (y_1,y_2)$, defined on the unit square $(-1/2,1/2)^2$ by (see figure \ref{check1}) \begin{equation} \; \sigma(y') = \left\{\!\! \begin{array}{l l}
\vspace{0.2cm}\alpha_1 & \text{in } Q_1:=(0,1/2)^2, \\
\vspace{0.2cm}\alpha_2 & \text{in } Q_2:=(0,1/2) \times (-1/2,0), \\
\vspace{0.2cm}\alpha_3 & \text{in } Q_3:=(-1/2,0)^2, \\
\vspace{0.2cm}\alpha_4 & \text{in } Q_4:=(-1/2,0) \times (0,1/2).
\end{array} \right. \;\;\;\label{defcheck}\end{equation}

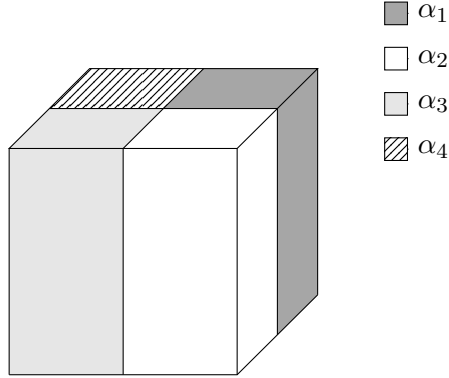
\begin{figure}[H]
\centering
\leavevmode
\begin{tikzpicture}[math3d, scale=3]
% face inferieure ABCD, D=O
\coordinate (a) at (1,0,0);
\coordinate (b) at (1,1,0);
\coordinate (c) at (0,1,0);
\coordinate (d) at (0,0,0);
\coordinate (b1) at (0.5,1,0);
\coordinate (a1) at (1,0.5,0);
% face superieure EFGH (obtenue par translation de ABCD)
\coordinate (hh) at (0.5,0.5,1);
\coordinate (e) at (1,0,1);
\coordinate (f) at (1,1,1);
\coordinate (g) at (0,1,1);
\coordinate (h) at (0,0,1);
\coordinate (e1) at (0.5,1,1);
\coordinate (f1) at (1,0.5,1);
\coordinate (g1) at (0,0.5,1);
\coordinate (h1) at (0.5,0,1);
\fill[color= gray!70] (b1) -- (c) -- (g) -- (g1) -- (hh) -- (e1) -- cycle;
\fill[color= gray!0] (h) -- (h1) -- (hh) -- (g1) -- cycle;
\fill[color= gray!20] (h1) -- (hh) -- (f1) -- (a1) -- (a) -- (e) -- cycle;
\fill[pattern=north east lines] (h) -- (h1) -- (hh) -- (g1) -- cycle;
\draw (e1)--(h1) -- cycle;
\draw (f1)--(g1) -- cycle;
\draw (a) -- (b) -- (c);
\draw (b1) -- (e1);
\draw (e) -- (f) -- (g) -- (h) -- cycle;
\draw (a1) -- (f1);
\draw (a) -- (e);
\draw (b) -- (f);
\draw (c) -- (g);
\fill[color= gray!70] (2,2,2) -- (2,2.1,2) -- (2,2.1,1.9) -- (2,2,1.9)-- cycle;
\draw (2,2,2) -- (2,2.1,2) -- (2,2.1,1.9) -- (2,2,1.9)-- cycle;
\draw[-] (2,2.1,1.95) -- (2,2.1,1.95)
node[midway, right]{$\alpha_1$};
\fill[color= gray!00] (2,2,1.8) -- (2,2.1,1.8) -- (2,2.1,1.7) -- (2,2,1.7)-- cycle;
\draw (2,2,1.8) -- (2,2.1,1.8) -- (2,2.1,1.7) -- (2,2,1.7)-- cycle;
\draw[-] (2,2.1,1.75) -- (2,2.1,1.75)
node[midway, right]{$\alpha_2$};
\fill[color= gray!20] (2,2,1.6) -- (2,2.1,1.6) -- (2,2.1,1.5) -- (2,2,1.5)-- cycle;
\draw (2,2,1.6) -- (2,2.1,1.6) -- (2,2.1,1.5) -- (2,2,1.5)-- cycle;
\draw[-] (2,2.1,1.55) -- (2,2.1,1.55)
node[midway, right]{$\alpha_3$};
\fill[pattern=north east lines] (2,2,1.4) -- (2,2.1,1.4) -- (2,2.1,1.3) -- (2,2,1.3)-- cycle;
\draw (2,2,1.4) -- (2,2.1,1.4) -- (2,2.1,1.3) -- (2,2,1.3)-- cycle;
\draw[-] (2,2.1,1.35) -- (2,2.1,1.35)
node[midway, right]{$\alpha_4$};

\end{tikzpicture}
\caption{Period cell of the four-phase checkerboard columnar structure}
\label{check1}
\end{figure}

We now state the following result:

\begin{Prop} Consider the conductivity defined by \eqref{defcheck} and $\mathscr D(h,h)$ by  \eqref{diffDronde} and assume that \eqref{condr} is satisfied. Then, $\mathscr D(h,h)=0$ if and only if one of the following conditions holds:\begin{itemize}
      \item $h=0$;
      \item $h \neq 0$ is not parallel to $e_i$ for $i=1,2,3$, and $\sigma$, $r$ are constant.
      \item $h = h_3 e_3 \neq 0$, and the Hall coefficient $r$ is constant;
      \item $h = h_i e_i \neq 0$ for $i=1,2$,  and  there exists a constant $C$ such that \eq \alpha_1 \, \alpha_3 = \alpha_2 \, \alpha_4, \quad \text{and} \quad r = C\left(\frac{\alpha_{6 - 2i}}{\alpha_1} \mathds{1}_{\{y_i > 0\}} + \mathds{1}_{\{y_i < 0\} }\right). \qe
      \end{itemize}
\label{check}
\end{Prop}

\begin{Rem} The case equality $\alpha_1 \, \alpha_3 = \alpha_2 \, \alpha_4$ corresponds to the case where the conductivity of the four-phase checkerboard is a tensor product of functions (see \cite{MarSpa}).

When $\alpha_1 \, \alpha_3 \neq \alpha_2 \, \alpha_4$, Craster and Obnosov \cite{Cras,CrasOb} proved an intricate formula for the corrector~$P^0$. In this case, $\sigma$ is not a tensor product of functions which is consistent with Proposition~\ref{cyl}.
\end{Rem}

\noindent \textit{Proof of Proposition \ref{check}.} The case $(h_1,h_2) = (0,0)$ and $h_3 \neq 0$ is a direct consequence of Proposition \ref{cyl}. Set, like in Proposition \ref{cyl}, $h' = (h_1,h_2)$.
\bigskip

\noindent \textit{First case: $h$ is not parallel to $e_i$ for $i=1,2,3$.} Assume that, without loss of generality, $|h'| = 1$, $h_i >0$ for $i=1,2$. We apply Proposition \ref{cyl}. There exist two positive functions $f$ and $g$ in $L^\infty(\R)$ which are $h_i$-periodic for $i=1,2$, and a constant $C$ such that \eqref{cylsigmafg} holds. Since $\sigma(y') = f(z_1) \, g(z_2)$ (with the new variables \eqref{chgcoor}) is a piecewise constant function, $f(z_1)$ and $g(z_2)$ are constant in each open square $Q_i$ for $i=1,2,3,4$. Considering the particular case of $Q_1$ and, for $\delta$ small enough, the rectangle \eq Q_{1,\delta} := \left\{ y' \in Q_1 \, : \, z_1 \in \left(\displaystyle \frac{h_1}{h_2} \delta, \displaystyle \frac{h_1+h_2}{2}\right), \, z_2 \in (-\delta, \delta) \right\} \subset Q_1,\qe we get successively that 
\eq \left\{ \begin{array}{l}
 \vspace{0.3cm} \text{$f$ is constant on } I := \displaystyle \left(0 , \frac{h_1 + h_2}{2}\right), \\
 \text{$g$ is constant on } J := \displaystyle \left(-\frac{h_1}{2} , \frac{h_2}{2}\right). \end{array} \right. \label{fgcsts}\qe Hence, $\sigma$ is constant in the rectangle (see figure \ref{interseccond}) \eq Q:= \left\{ y' \in \left(-\frac{1}{2}, \frac{1}{2} \right) \, : \, (z_1,z_2) \in I \times J \right\}. \qe  \begin{figure}[H]
\centering
\leavevmode
\begin{tikzpicture}[scale=5]
\draw[->] (0.5,-0.2) -- (0.5,1.2);
\draw[->] (-0.2,0.5) -- (1.2,0.5);
\draw[->] (0.5,0.5) -- (0.75,0.93)
node[midway,right]{$h'$};
%\draw[->] (0.5,0.5) -- (0.07,0.75)
%node[midway,right]{$Jh'$};
\draw[<-,dash pattern = on 1mm off 1mm] (0.85,1.1) -- (0.15,-0.1);
\draw[<-,dash pattern = on 1mm off 1mm] (-0.102,0.85) -- (1.102,0.15);
\coordinate (a) at (0,0);
\coordinate (b) at (1,0);
\coordinate (c) at (1,1);
\coordinate (d) at (0,1);
\coordinate (e) at (0.199, 0.675);
\coordinate (f) at (0.715, 0.375);
\coordinate (g) at (0.7850, 0.4954);
\coordinate (h) at (0.2690, 0.7954);
\coordinate (a1) at (0.5,0);
\coordinate (b1) at (1,0.5);
\coordinate (c1) at (0.5,1);
\coordinate (d1) at (0,0.5);
\fill[color= gray!20,opacity=0.5] (e) -- (f) -- (g)  -- (h) -- cycle;
\draw (e) -- (f) -- (g)  -- (h) -- cycle;
\draw (0.90,0.90) node{$Q_1$};
\draw (0.9,0.1) node{$Q_2$};
\draw (0.1,0.1) node{$Q_3$};
\draw (0.1,0.90) node{$Q_4$};
\draw (0.6,1.1) node{$y_2$};
\draw (1.1,0.6) node{$y_1$};
\draw (0.9,1.1) node{$z_1$};
\draw (-0.15,0.8) node{$z_2$};
\draw (1.35,1.025) node{$Q$};
\draw (1.38,0.925) node{$Q_{1,\delta}$};
\draw (-0.4,1.025) node{$\textcolor{white}{Q}$};
\draw (a) -- (b) -- (c) -- (d) -- cycle;
\draw (0.5035, 0.5555) -- (0.5465, 0.5305) -- (0.5915, 0.6079) -- (0.5485, 0.6329) -- cycle;
\fill [color= black] (0.5035, 0.5555) -- (0.5465, 0.5305) -- (0.5915, 0.6079) -- (0.5485, 0.6329) -- cycle;
\fill [color= black] (1.2,0.9) -- (1.25,0.9) -- (1.25,0.95) -- (1.2,0.95) -- cycle;
\draw (1.2,0.9) -- (1.25,0.9) -- (1.25,0.95) -- (1.2,0.95) -- cycle;
\fill [color= gray!20,opacity=0.5] (1.2,1) -- (1.25,1) -- (1.25,1.05) -- (1.2,1.05) -- cycle;
\draw (1.2,1) -- (1.25,1) -- (1.25,1.05) -- (1.2,1.05) -- cycle;
\end{tikzpicture}
\caption{Cross-section of the period cell of the checkerboard in the $(y_1,y_2)$-plane}
\label{interseccond}
\end{figure}
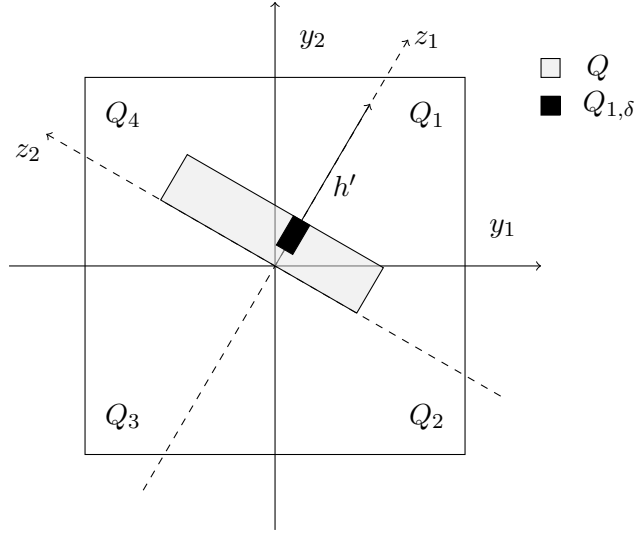 As $Q$ intersects $Q_1$, $Q_2$, $Q_4$, $\sigma$ is constant in $Q_1 \cup Q_2 \cup Q_4$, hence $\alpha_1 = \alpha_2 = \alpha_4$. Repeating the same argument with $Q_3$ in \eqref{fgcsts}, we obtain that $\alpha_1 = \alpha_2 = \alpha_3 =\alpha_4$. Therefore, $\sigma$, $f$, $g$ are constant, so is $r$ by \eqref{cylsigmafg}.

\bigskip

\noindent \textit{Second case: $h = h_1 e_1 \neq 0$.}  Without loss of generality, assume that $h_1 = 1$. We again apply Proposition \ref{cyl}. There exist two positive functions $f$ and $g$ in $L^\infty(\R)$ which are $h_i$-periodic for $i=1,2$, and a constant $C'$ such that \eq  \sigma(y') = f( y_1) \, g(y_2), \quad \text{and} \quad r(y') = \frac{C'}{f(y_1)} \quad \text{a.e. } y' \in (-1/2,1/2)^2.\label{eq12}\qe Since $\sigma$ is piecewise constant in the four-phase checkerboard and $f$, $g$ are respectively functions of the independent variables $y_1$, $y_2$, $f$, $g$ are constant in each open square $Q_i$, for $i=1,2,3,4$. It follows immediately that there exist two positive constants $C'_1$ and $C'_2$ such that $C'_1 \, C'_2 = \alpha_1^{-1}$ and \begin{equation} \displaystyle \left\{\!\! \begin{array}{r l}
\vspace{0.3cm} f(y_1) & = \displaystyle C'_1\left( \alpha_1 \mathds{1}_{\{y_1 >0\}} + \alpha_4 \mathds{1}_{\{y_1 \leq 0\}} \right), \\
g(y_2) & = \displaystyle C'_2 \left( \alpha_1 \mathds{1}_{\{y_2 >0\}} + \alpha_2 \mathds{1}_{\{y_2 \leq 0\}} \right),
\end{array} \quad \text{a.e. } y' \in (-1/2,1/2)^2.\right. \;\;\;\label{eqfg}\end{equation} Finally, \eqref{eqfg} and \eqref{eq12} imply  that there exists a constant $C$ such that \eq \alpha_1 \, \alpha_3 = \alpha_2 \, \alpha_4, \quad \text{and} \quad r(y') = C\left(\frac{\alpha_4}{\alpha_1} \mathds{1}_{\{y_1 > 0\}} + \mathds{1}_{\{y_1 < 0\} }\right).\label{concr}\qe Conversely, if \eqref{concr} is satisfied, we can define $f$ and $g$ as in \eqref{eqfg} with $C'_1 C'_2 = \alpha_1^{-1}$.

The case $h = h_2 e_2 \neq 0$ is quite similar. \qed

\bibliographystyle{plain}
\bibliography{biblio}

\end{document}